\documentclass[12pt]{article}
\usepackage{amsfonts}
\usepackage{amsmath}
\usepackage{amssymb}
\usepackage[mathscr]{eucal}
\usepackage{graphicx}
\usepackage{soul,xcolor}
\usepackage{fullpage}
\usepackage{times}
\usepackage{bbm}
\usepackage{epsfig}
\usepackage{color}
\usepackage[pagebackref]{hyperref}    
\usepackage{caption}

\def\eod{\vrule height 6pt width 5pt depth 0pt}
\newenvironment{proof}{\noindent {\bf Proof:} \hspace{.2em}}
                      {\hspace*{\fill}{\eod}}

\newtheorem{theorem}{Theorem}
\newtheorem{lemma}[theorem]{Lemma}
\newtheorem{definition}[theorem]{Definition}
\newtheorem{corollary}[theorem]{Corollary}

\newtheorem{example}[theorem]{Example}
\newtheorem{remark}[theorem]{Remark}

\newcommand{\ED}{ \mathsf{ED}}

\newcommand{\sL}{  \mathcal{ L}}
\newcommand{\sM}{  \mathcal{ M}}
\newcommand{\ssM}{  \mathsf{Moment}}
\newcommand{\sO}{  \mathcal{ O}}

\newcommand{\RR}{ \mathbbm{R}}
\newcommand{\ZZ}{ \mathbbm{Z}}

\newcommand{\revpath}{ \mathsf{reverse\_on\_path}}

\newcommand{\rhalf}{\lfloor r/2 \rfloor}
\newcommand{\nhalf}{\lfloor n/2 \rfloor}

\newcommand{\comment}[1]{}
\newcommand{\red}[1]{\textcolor{red}{#1}}

\newcommand{\rar}{ \rightarrow}
\newcommand{\SSS}{\mathfrak{S}}

\newcommand{\CC}{ \mathbb{C}}

\newcommand{\aw}{\mathsf{away}}
\newcommand{\GTS}{\mathsf{GTS}}

\newcommand{\awy}{\mathsf{Aw}}
\newcommand{\wt}{\mathsf{wt}}

\newcommand{\bd}{\mathsf{bidir}}
\newcommand{\fr}{\mathsf{free}}

\newcommand{\perm}{  \mathsf{perm}}

\newcommand{\vsp}{\vskip 1em}

\usepackage{mathtools}

\begin{document}

\setstcolor{red}

\title{Laplacian Immanantal polynomials and the $\GTS$ poset on Trees
}

\author{Mukesh Kumar Nagar\\
Department of Mathematics\\
Indian Institute of Technology, Bombay\\
Mumbai 400 076, India.\\
email: mukesh.kr.nagar@gmail.com
\and
Sivaramakrishnan Sivasubramanian\\
Department of Mathematics\\
Indian Institute of Technology, Bombay\\
Mumbai 400 076, India.\\
email: krishnan@math.iitb.ac.in
}

\maketitle

 \vskip 2em

{\it {\bf Abstract:} Let $T$ be a tree on $n$ vertices with Laplacian $L_T$ and
let $\GTS_n$ be the generalized tree shift poset on the set of unlabelled 
trees on $n$ vertices.  Inequalities are known for coefficients of the
characteristic polynomial of $L_T$ as we go up the poset $\GTS_n$.
In this work, we generalize these inequalities to the 
$q$-Laplacian $\sL^q_T$ of $T$ and to the coefficients of all immanantal
polynomials.
}

  \vsp

   \vsp


	  \vsp


\section{Introduction}
\label{sec:intro}

Csikv{\'a}ri in \cite{csikvari-poset1} defined a poset 
on the set of unlabelled trees with $n$ vertices that we
denote in this paper as $\GTS_n$.  Among other results, he showed that 
going up on $\GTS_n$ has the following effect:  the coefficients of the 
characteristic polynomial of the Laplacian $L_T$ of $T$ decrease in absolute value.   
In this paper, we prove the following more general result about
immanantal polynomials of the $q$-Laplacian matrix of trees 

\begin{theorem}
  \label{thm:main}
Let $T_1$ and $T_2$ be trees with $n$ vertices and let $T_2$ cover $T_1$ in $\GTS_n$.  Let
$\sL_{T_1}^q$ and $\sL_{T_2}^q$ be the $q$-Laplacians of $T_1$ and $T_2$ respectively.  
For $\lambda \vdash n$, let 

\begin{eqnarray*}
f^{\sL_{T_1}^q}_{\lambda}(x) & = & d_{\lambda}(xI -\sL_{T_1}^q) = \sum_{r=0}^n (-1)^r 
c_{\lambda,r}^{\sL_{T_1}^q}(q) x^{n-r} \mbox{ and} \\
f^{\sL_{T_2}^q}_{\lambda}(x) & = & d_{\lambda}(xI -\sL_{T_2}^q) = \sum_{r=0}^n (-1)^r 
c_{\lambda,r}^{\sL_{T_2}^q}(q) x^{n-r}.
\end{eqnarray*}

Then, for all $\lambda \vdash n$ and for all $0 \leq r \leq n$, we assert that 
$c_{\lambda,r}^{\sL_{T_1}^q}(q) - c_{\lambda,r}^{\sL_{T_2}^q}(q) 
\in \RR^+[q^2]$.   
\end{theorem}

For a positive integer $n$, let $[n] = \{1,2,\ldots,n\}$.
Let $\SSS_n$ be the group of permutations of $[n]$.
Let $\chi_{\lambda}^{}$ be the irreducible character of the 
$\SSS_n$ over $\CC$ indexed by the partition $\lambda$ 
of $n$.  We refer the reader to the book by Sagan \cite{sagan-book}
as a reference for results on representation theory that we 
use in this work.
We denote partitions $\lambda$ of $n$ as $\lambda \vdash n$.  This means
we have $\lambda = \lambda_1, \lambda_2, \ldots, \lambda_l$ where $\lambda_i \in \ZZ$ 
for all $i$ with  $\lambda_1 \geq \lambda_2 \geq \cdots \geq \lambda_l > 0$ and with
$\sum_{i=1}^l \lambda_i = n$. We also write partitions using the exponential notation, 
with multiplicities of parts written as exponents.  Since characters of $\SSS_n$ are 
integer valued,  
we think of $\chi_{\lambda}^{}$ as a function $\chi_{\lambda}^{}: \SSS_n \rar \ZZ$. 
Let $\lambda \vdash n$ and let 
$A = (a_{i,j})_{1\leq i,j \leq n}$ be an $n \times n$ matrix.  Define its immanant as
$\displaystyle d_{\lambda}(A) = 
\sum_{\psi \in \SSS_n} \chi_{\lambda}^{}(\psi) \prod_{i=1}^n a_{i,\psi_i}.$
It is well known that $d_{1^n}(A)= \det(A)$ and $d_{n}(A)= \perm(A)$ where
$\perm(A)$ is the permanent of $A$.  

For an $n \times n$ matrix $A$, define 
$f_{\lambda}^A(x)= \displaystyle d_{\lambda}(xI-A)$.  
The polynomial $f_{\lambda}^A(x)$ is called the immanantal polynomial of $A$ corresponding to 
$\lambda \vdash n$. 
Thus, in this notation, $f_{1^n}^A(x)$ is the characteristic polynomial of $A$.  
Let $T$ be a tree with $n$ vertices with Laplacian matrix $L_T$ and define 

\begin{equation}
\label{eqn:lapl_immanant_poly}
f^{L_T}_{\lambda}(x)= d_{\lambda}(xI - L_T) =  
\sum_{r=0}^{n} (-1)^r c_{\lambda,r}^{L_T} x^{n-r} 
\end{equation}

where the $c_{\lambda,r}^{L_T}$'s are coefficients of the Laplacian 
immanantal polynomial of $T$  in absolute value.  Immanantal 
polynomials were studied by Merris 
\cite{merris-second-imm-polynom} where the Laplacian immanantal polynomial
corresponding to the partition $\lambda = 2,1^{n-2}$ (also 
called the second immanantal polynomial) of a 
tree $T$  was shown to have connections with the centroid of $T$.
Botti and Merris \cite{botti-merris-almost_all_trees} showed
that almost all trees share a complete set of Laplacian immanantal polynomials.
When $\lambda=1^n$, Gutman and Pavlovic 
\cite{gutman-pavlovic_Laplacian_coeff} conjectured the following 
inequality  
which was proved by Gutman and Zhou 
\cite{gutman-zhao-connection-Laplacian-spectra}
and independently by Mohar \cite{mohar-Laplacian-coeffs-acyclic-graphs}.

\begin{theorem}[Gutman and Zhou, Mohar]
\label{thm:Gutman_pavlovic_conjecture}
Let $T$ be any  tree on $n$ vertices and let $S_n$ and $P_n$ be the star and the path
trees on $n$ vertices respectively.  Then, for $0 \leq r \leq n$, we have
\begin{equation*}
  c_{1^n,r}^{L_{S_n}} \leq c_{1^n,r}^{L_T} \leq c_{1^n,r}^{L_{P_n}}.
\end{equation*}
\end{theorem}

Thus, in absolute value, any tree $T$ has coefficients of its Laplacian characteristic 
polynomial sandwiched between the corresponding coefficients of the star and the path 
trees.  
Mohar actually proves stronger inequalities than this result, see 
Csikv{\'a}ri  \cite[Section 10]{csikvari-poset2} for information on Mohar's
stronger results.  Much earlier, Chan, Lam and Yeo in their preprint 
\cite{chan_lam_yeo}, proved the following.

\begin{theorem}[Chan, Lam and Yeo]
  \label{thm:chan_lam_yeo}
Let $T$ be any  tree on $n$ vertices with Laplacian $L_T$ and let $S_n$ and $P_n$ 
be the star and the path
trees on $n$ vertices respectively.  Then, 
for all $\lambda \vdash n$ and $0\leq r \leq n$, 
\begin{equation}
  \label{eqn:chan_lam_yeo}
  c_{\lambda,r}^{L_{S_n}} \leq c_{\lambda,r}^{L_T} \leq c_{\lambda,r}^{L_{P_n}}.
\end{equation}
\end{theorem}

In this work, we consider the $q$-Laplacian matrix $\sL_T^q$ of a tree $T$ on $n$ vertices.  
It is defined as $\sL_T^q = I + q^2(D-I) - qA$ where $q$ is a variable, 
$D$ is the diagonal matrix with degrees on the diagonal and $A$ is the adjacency 
matrix of $T$.  $\sL_T^q$ can be defined for arbitrary graphs $G$ analogously
and it is clear that when $q=1$, $\sL_G^q = L_G$.  The matrix $\sL_G^q$ has 
occurred previously in connection with the Ihara-Selberg zeta function of $G$
(see Bass \cite{bass} and Foata and Zeilberger 
\cite{foata-zeilberger-bass-trams}). 
For trees, $\sL_T^q$ has connections with the inverse of $T$'s exponential
distance matrix (see Bapat, Lal and Pati \cite{bapat-lal-pati}).  As done in
\eqref{eqn:lapl_immanant_poly}, define

\begin{equation}
  \label{eqn:q-lapl_immanant_poly}
  f^{\sL_T^q}_{\lambda}(x) = d_{\lambda}(xI -\sL_T^q) = \sum_{r=0}^n (-1)^r c_{\lambda,r}^{\sL_T^q}(q) x^{n-r}.
\end{equation}

We consider the following counterpart of inequalities like 
\eqref{eqn:chan_lam_yeo} when each 
coefficient is a polynomial in the variable $q$:  we want the difference
$c_{\lambda,r}^{\sL_T^q}(q) - c_{\lambda,r}^{\sL_{S_n}^q}(q) \in \RR^+[q]$.  That
is, the difference polynomial has only positive coefficients.  This is 
the standard way to get $q$-analogue of inequalities.  Similarly, we want
$c_{\lambda,r}^{\sL_{P_n}^q}(q) - c_{\lambda,r}^{\sL_{T}^q}(q) \in \RR^+[q]$.

We mention a few lines about our proof of Theorem \ref{thm:main}.  In 
\cite[Theorem 5.1]{csikvari-poset2}, Csikv{\'a}ri gives a ``General Lemma'' from
which he infers properties about polynomials associated to trees.  
In that lemma, the following crucial property is needed when 
dealing with characteristic polynomials of matrices.  Let $M= A \oplus B$
be an $n \times n$ matrix that can be written as a direct sum of two 
square matrices.  Then, clearly $\det(M) = \det(A) \det(B)$.  
This property is sadly not true for other immanants.  That is, 
$d_{\lambda}(M) \not= d_{\lambda}(A) d_{\lambda}(B)$ 
(indeed, the definition of 
$d_{\lambda}(A)$ is not clear 
when $\lambda \vdash n$ and $A$ is an $m \times m$ matrix with $m < n$).  
We thus combinatorialise the immanant as done by Chan, Lam and
Yeo \cite{chan_lam_yeo} and express the immanantal polynomial
in terms of matchings and vertex orientations.  
Section \ref{sec:gts_poset} gives preliminaries on the $\GTS_n$
poset and Section \ref{sec:Bmatchings} gives the necessary background on
$B$-matchings, $B$-vertex orientations and their connection
to coefficients of immanantal polynomials.  We give our proof of  Theorem \ref{thm:main}
in Section \ref{sec:coeff_imman_poly} and draw several corollaries
in Sections \ref{sec:centroid}, 
\ref{sec:qt_laplacian} 
and 
\ref{sec:expon_dist_mat} 
involving the $q^2$-analogue of 
vertex moments in a tree, $q,t$-Laplacian matrices which
include the Hermitian Laplacian of $T$ and $T$'s 
exponential distance matrices.

\section{The poset $\GTS_n$} 
\label{sec:gts_poset}
Though Csikv{\'a}ri in \cite{csikvari-poset1} defined the poset
on unlabelled trees with $n$ vertices, we will label
the vertices of the trees according to some convention (see Remark \ref{rem:label}).
We recall the definition of this poset.

\begin{figure}[h]
\centerline{\includegraphics[scale=0.55]{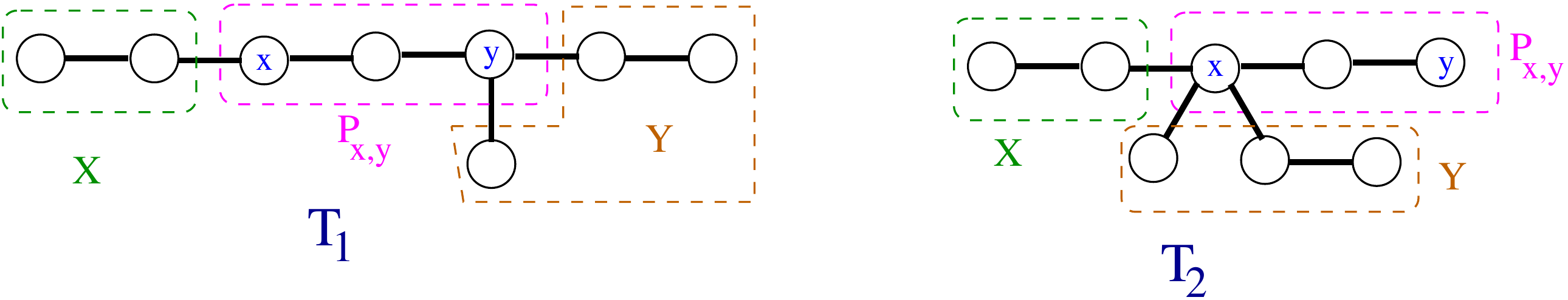}}
\caption{Two trees with $T_2 \geq_{\GTS_n} T_1$ and $T_2$ covering $T_1$.}
\label{fig:gts_example}
\end{figure}

\begin{definition}
  Let $T_1$ be a tree on $n$ vertices and $x, y$ be two vertices of $T_1$.  Let $P_{x,y}$ be the
  unique path in $T_1$ between $x$ and $y$.  Assume that $x$ and $y$ are such that 
all the interior vertices (if they exist) on $P_{x,y}$ have degree 2.  
Let $z$ be the neighbour of $y$ on the path $P_{x,y}$.  Consider the tree $T_2$ obtained by 
moving all neighbours of $y$ except $z$ to the vertex $x$. 
This is illustrated in Figure \ref{fig:gts_example}.  This move helps us
to partially order the set of unlabeled trees on $n$ vertices.   We 
denote this poset on trees with $n$ vertices as $\GTS_n$.  
We say $T_2$ is above $T_1$ in $\GTS_n$ or that $T_1$ is below $T_2$ in $\GTS_n$
and denote it as $T_2 \geq_{\GTS_n} T_1$.
The poset $\GTS_6$ is illustrated in Figure \ref{fig:gts6}.
\end{definition}

If $T_2 \geq_{\GTS_n} T_1$ and there is no tree $T$ with $T \not= T_1, T_2$ 
such that $T_2 \geq_{\GTS_n} T \geq_{\GTS_n} T_1$, then we say $T_2$ covers
$T_1$ (see Figure \ref{fig:gts_example}).
If either $x$ or $y$ is a leaf vertex in $T_1$, then it is easy to check that $T_2$ is 
isomorphic to $T_1$.  If neither $x$ nor $y$ is a leaf in $T_1$, then $T_2$ 
is said to be obtained from $T_1$ by a proper generalized tree shift (PGTS henceforth).
Clearly, if $T_2$ is obtained by a PGTS from $T_1$, then, the number 
of leaf vertices of $T_2$ is one more than the number of leaf vertices of $T_1$. 
Csikv{\'a}ri in \cite{csikvari-poset1} showed the following.

\begin{lemma}[Csikv{\'a}ri]
  \label{lem:csikvari-prelim}
Every tree $T$ with $n$ vertices other than the path, lies above some 
other tree $T'$ on $\GTS_n$.  The star tree on $n$ vertices is the maximal element and
the path tree on $n$ vertices is the minimal element of $\GTS_n$.
\end{lemma}

\begin{figure}[h]
\centerline{\includegraphics[scale=0.35]{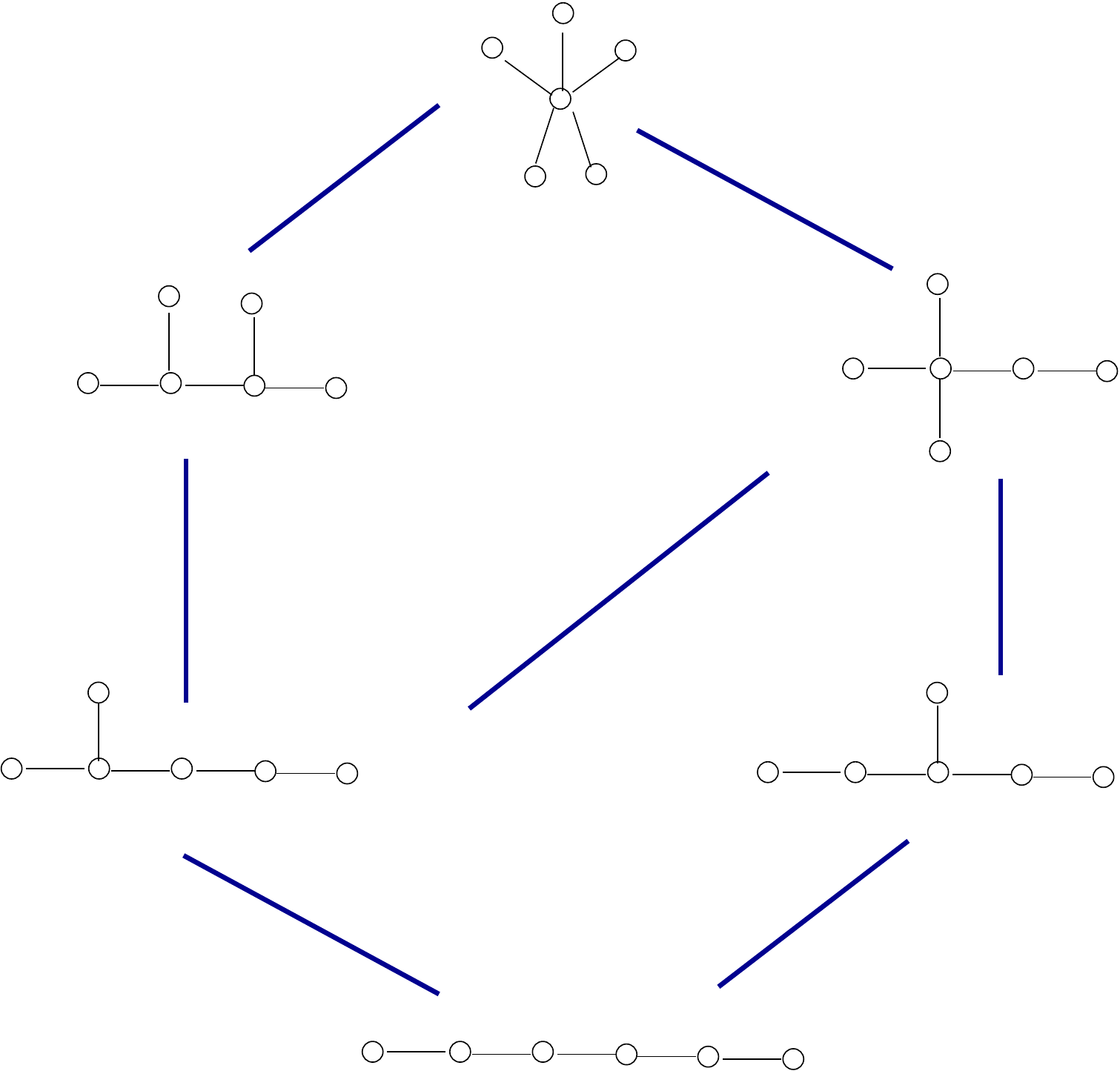}}
\caption{The poset $\GTS_6$ on trees with 6 vertices.}
\label{fig:gts6} 
\end{figure}

\section{$B$-matchings and $B$-vertex orientations}
\label{sec:Bmatchings}
As done in earlier work \cite{mukesh-siva-hook}, we use matchings in $T$ to index
terms that arise in the computation of the immanant $d_{\lambda}(\sL_T^q)$.
A dual concept of vertex orientations was used to get a near positive expression
for immanants of $\sL_T^q$.

In this work, we need to find $f_{\lambda}^{\sL_T^q}(x) = d_{\lambda}(xI - \sL_T^q)$.  
As done by Chan, Lam and Yeo \cite{chan_lam_yeo}, we 
index terms that occur in the computation of $f_{\lambda}^{\sL_T^q}(x)$ by 
partial matchings that we term as $B$-matchings.   
Let $T$ have vertex set $V$ and edge set $E$.  Let $B \subseteq V$ with 
$|B| = r$ and let $F_B$ be the forest induced by $T$ on the set $B$.  
A $B$-matching of $T$ is a subset $M \subseteq E(F_B)$ of 
edges of $F_B$ such that each vertex $v \in B$ is adjacent to at most one edge in $M$.  
If the number of edges in $M$ equals $j$, then $M$ is called a $j$-sized
$B$-matching in $T$.  Let $\sM_j(B)$ denote the set of $j$-sized $B$-matchings in $T$.  
For vertex $v$, we  denote its degree $\deg_{T}(v)$ in $T$  
alternatively as $d_v$.  For $M \in \sM_j(B)$, define a polynomial 
weight $\wt_{B,M}(q) =q^{2j} \displaystyle \prod_{v \in B-M} [1+q^2(d_v-1)]$. 
Define 
\begin{equation*}
m_{B,j}^{}(q) = \displaystyle \sum_{M \in \sM_j(B)} \wt_{B,M}(q)
\mbox{  and  } \ \ m_{r,j}^{}(q) =\displaystyle \sum_{B \subset V, |B|=r} m_{B,j}^{}(q).
\end{equation*}

Define $\chi_{\lambda}^{}(j)$ to be the character 
$\chi_{\lambda}^{}( \cdot )$ evaluated at such a permutation
with cycle type $2^j, 1^{n-2j}$.  The following lemma is 
straightforward from the definition of immanants.

\begin{lemma}
\label{lem:coeff_in_m_rj}
Let $T$ be a tree on vertex set $[n]$ with $q$-Laplacian $\sL_T^q$.  
Let $\lambda \vdash n$
and let $0 \leq r \leq n$. Then, the coefficient $c_{\lambda,r}^{\sL_T^q}(q)$ 
as defined in \eqref{eqn:q-lapl_immanant_poly} equals
\begin{equation*}
  c_{\lambda,r}^{\sL_T^q}(q) = \sum\limits_{j=0}^{\rhalf} \chi_{\lambda}^{}(j)m_{r,j}^{}(q).
\end{equation*}
\end{lemma}

\begin{proof}
Let $B\subseteq [n]$ with $|B|=r$. Then, clearly
$ 
 c_{\lambda,B}^{\sL_T^q}(q)=d_{\lambda}
 \left[
 \begin{array}{c c} 
 \sL_T^q[B|B] & 0 \\
 0 & I
 \end{array}
 \right]$, 
where $\sL_T^q[B|B]$ is the sub-matrix of $\sL_{T}^q$ 
induced on the rows and columns with indices in the set $B$ and 
$I$ is the $n-r \times n-r$ identity matrix. Further,
it is clear that $c_{\lambda,r}^{\sL_T^q}(q) = 
\sum_{B \subseteq [n], |B| = r} c_{\lambda,B}^{\sL_T^q}(q)$.

Note that there is no cycle in $T$, and hence
in the forest $F_B$. 
Thus, each permutation $\psi \in \SSS_n$ which
in cycle notation has a cycle of length strictly greater than $2$, 
will satisfy $\prod_{i=1}^n \ell_{i,\psi_i}=0$. 
Therefore, only permutations $\psi \in \SSS_n$ 
which fix the set $[n]-B$ and have cycle type
$2^j, 1^{n-2j}$  contribute to $c_{\lambda,B}^{\sL_T^q}(q)$.  
It is easy to see that such permutations 
can be identified with $j$-sized $B$-matchings in $F_B$ 
and that  this correspondence is reversible. 

 Recall $\sM_j(B)$ is the set of $j$-sized $B$ matchings in $T$.  
Clearly, the contribution to $c_{\lambda,B}^{\sL_T^q}(q)$ from
permutations which fix $[n]-B$ and have cycle-type $2^j, 1^{n-2j}$ 
is $\chi_{\lambda}^{ }(j) m_{B,j}^{}(q)$.
Thus, we see that
\begin{equation}
\label{eq:coeff_c_in_B} 
c_{\lambda,B}^{\sL_T^q}(q) 
= \sum_{j=0}^{\rhalf} \chi_{\lambda}^{ }(j) m_{B,j}^{}(q).
\end{equation}

Summing over various $B$'s of size $r$ completes the proof.
\end{proof}

\subsection{$B$-vertex orientations}
As done by Chan, Lam and Yeo \cite{chan_lam_yeo}, we next express coefficients 
of the immanantal polynomial as a sum of almost positive summands where
the summands are indexed by partial vertex orientations that
we term as $B$-vertex orientations.

Let $T$ be a tree with vertex set $V = [n]$.
For $B \subseteq [n]$, we orient each vertex $v \in B$ to one of its neighbours 
(which may or may not be in $B$).  Such vertex orientations are termed
as $B$-vertex orientations.  Let $O$ be a $B$-vertex orientation.  Each $v \in B$ has 
$d_v$ orientation choices.  We depict the orientation $O$ in pictures
by drawing an arrow on the edge from $v$ to its oriented 
neighbour and directing the arrow away from $v$.  
We do not distinguish between $O$ and its picture 
from now on.  In $O$, edges thus get arrows and 
there may be edges 
which have two arrows,  one in each direction (see Figures
\ref{fig:case0}, \ref{fig:case2prime} and \ref{fig:last_map_example}
for examples).  We call such edges as 
{\it bidirected arcs} and let $\bd(O)$ denote the set of 
bidirected arcs in $O$. We extend this notation to vertices $v \in B$ and say
$v \in \bd(O)$ if $\{u,v\} \in \bd(O)$ for some $u \in B$.  We also 
say $v \in B$ is {\sl free} in $O$ if $v \in B - \bd(O)$ and denote
by $\fr(O)$ the set of free vertices of $O$.

In $T$, let $\sO_{B,i}^T$ be the set of $B$-orientations 
$O$, such that $O$ has $i$ bidirected arcs.   We need to separate 
the case $B = V$ from the cases $B \neq V$.  First, let $B \neq V$.  
For such a $B \subseteq V$, let $m = \min_{v \in [n]-B}^{} v$ be the minimum
numbered vertex outside $B$ and let $O \in \sO_{B,i}^T$.  For each 
$v \in \fr(O)$, as there is a unique path from $v$ to $m$ in $T$,
we can tell if $v$ is oriented {\it ``towards''} $m$ or if $v$ is 
oriented {\it ``away from''} $m$.  Formally, for $O \in \sO_{B,i}^T$, 
define a 0/1 function $\aw: \fr(O) \rightarrow \{0,1\}$ by 

$$\aw(v)=\left\{
\begin{array}{l l}
1 &\mbox{ if $v$ is oriented away from } m, \\
0 & \mbox{if $v$ is oriented towards } m. \\
\end{array}  
\right.$$ 

For each $O\in \sO_{B,i}^T$ assign the following non-negative integer: 
$$\awy_B^{T}(O) = 2i+2\sum_{v\in \fr(O)} \aw(v).$$

Define the generating function of the statistic $\awy_B^{T}(\cdot)$ in the
variable $q$ as follows:

\begin{eqnarray}
\label{eq:def_a_B,i}
a_{B,i}^T(q) & = & \sum_{O \in \sO_{B,i}^T} q^{\awy_B^{T}(O)}, \\
\label{defn:ais_sum}
a_{r,i}^T(q) & = & \sum_{B \subset V, |B|=r} a_{B,i}^T(q) = 
\sum_{B \subseteq V, |B| = r}  \sum_{O \in \sO_{B,i}^T} q^{\awy_B^T(O)}.
\end{eqnarray}

\begin{example}
  Let $T_2$ be the tree given in Figure 
\ref{fig:gts_label_example}
  and let $B = \{2,4,6,7,8\}$
  with $|B| = r = 5$.  Below we give $a_{B,i}^{T_2}(q)$ for $i$ from 0 to $\rhalf$.

\begin{center}
	 $\begin{array}{|c|c|c|c|}  \hline
	i & 0 & 1 & 2   \\ \hline
	a_{B,i}^{T_2}(q) & 1+2q^2+q^4 & q^2(1+2q^2+q^4) & 0 \\ \hline
  \end{array}$
\end{center}
\end{example}

\begin{remark}
  \label{rem:poly_q_sq}
For any tree $T$ and all $r,j$, it is easy to see from the definitions
that both $m_{r,j}^{}(q)$ and $a_{r,i}^T(q)$ are polynomials in $q^2$. 
\end{remark}

Chan, Lam and Yeo in \cite{chan_lam_yeo} showed for the Laplacian 
$L_T$ of a tree $T$, a relation involving numerical counterparts of $m_{B,j}^{}(q)$'s
and $a_{B,i}^T(q)$'s.  Chan and Lam \cite{hook_immanant_explained-chan_lam}
had already proved this identity 
for the special case when $B = [n]$.  Earlier, we had in 
\cite[Theorem 11]{mukesh-siva-hook}  
obtained a $q$-analogue of this identity when $B = [n]$.
There, care had to be taken to define 
$a_{[n],0}^T(q) = 1-q^2$.  We give a $q$-analogue  below in Lemma 
\ref{lem:reln_aiq_mjq} when $B$ can be an arbitrary subset.  
In \cite{mukesh-siva-hook}, since $B = [n]$, there was no
vertex outside $B$ and hence $m$ could not be defined.  There, the lexicographically
minimum edge of the matching $M$ was used in place of $m$.  It is easy to
see that we could have used the lexicographically minimum edge of $M$ 
when $B \neq [n]$ as well.  
From now onwards, we are free from this restriction $B \not= [n]$.
Since the proof is identical to that of 
\cite[Theorem 11]{mukesh-siva-hook}, we omit it and merely
state the result.

\begin{lemma}
  \label{lem:reln_aiq_mjq}
Let $T$ be a tree with vertex set $[n]$ and $B$ be an $r$-subset of $[n]$. Then,  
$$m_{B,j}^{}(q)=\sum\limits_{i=j}^{\rhalf} {i \choose j} a_{B,i}^T(q).  \mbox{  Moreover, }
m_{r,j}^{}(q)=\sum\limits_{i=j}^{\rhalf} {i \choose j} a_{r,i}^T(q).$$
\end{lemma}

Chan and Lam in \cite{chan-lam-binom-coeffs-char} 
showed the following non-negativity result on characters summed 
with binomial coefficients as weights.  
Let $n \geq 2$ and let $\lambda \vdash n$.  
Recall $\chi_{\lambda}^{}(j)$ is the 
character $\chi_{\lambda}^{}$ evaluated at a permutation with cycle type 
$2^j, 1^{n-2j}$.

\begin{lemma}[Chan and Lam]
  \label{lem:character_sum}
Let $\lambda \vdash n$ and let $\chi_{\lambda}^{}(j)$ be as defined above.  
Let $0 \leq i \leq \nhalf$.
Then
$\sum_{j=0}^i \chi_{\lambda}^{}(j) {i \choose j} =  \alpha_{\lambda,i}^{} 2^i,$
where $\alpha_{\lambda,i}^{} \geq 0$.  Further, if 
$\lambda = k,1^{n-k}$, then $\alpha_{\lambda,i}^{ } = \binom{n-i-1 }{k-i-1}$.
\end{lemma}

Combining Lemmas 
\ref{lem:reln_aiq_mjq} 
and 
\ref{lem:character_sum} 
with
Lemma \ref{lem:coeff_in_m_rj}
gives us the following Corollary whose
proof we omit.
This gives an interpretation of the coefficient $c_{\lambda,r}^{\sL_T^q}(q)$ 
in the immanantal polynomial 
as a functions of the $a_{r,i}^T(q)$'s.   Since all the $a_{r,i}^T(q)$'s except
$a_{[n],0}^T(q)$ have positive coefficients, this is an almost positive expression.

\begin{corollary}
\label{cor:coeff_interpret} 
For $0 \leq r \leq n$, the coefficient of the immanantal polynomial of 
$\sL_T^q$  in absolute value is given by
\begin{equation*}
  c_{\lambda,r}^{\sL_T^q}(q) = \sum\limits_{i=0}^{\rhalf} 
  \alpha_{\lambda,i}^{ } 2^i a_{r,i}^T(q),  \
  \mbox{ where } \alpha_{\lambda,i}^{ } \geq 0, \>\> \forall \ \lambda\vdash n, i.
\end{equation*} 
\end{corollary}

Combining \eqref{eq:coeff_c_in_B}, Lemmas \ref{lem:character_sum} and \ref{lem:reln_aiq_mjq} 
gives us another corollary when the 
partition is $\lambda = 1^n$, which we
again merely state.

\begin{corollary}
\label{cor:smaller_dets} 
When $\lambda = 1^n$, we have 
$\alpha_{\lambda,i}^{} = \begin{cases} 
		1 & \mbox{ if } i = 0\\
		0 & \mbox{ otherwise }
  	\end{cases}
$.  Further, let $B \subseteq [n]$ with $|B| = r$.   Then, 
\begin{equation*}
 \det(\sL_T^q[B|B])  =  a_{B,0}^T(q). \mbox{ Moreover, }  c_{1^n,r}^{\sL_T^q}(q) = a_{r,0}^T(q).
\end{equation*} 
\end{corollary}

\begin{remark}
Let tree $T$ have vertex set $[n]$ and let $B\subseteq [n]$ with $|B|=n-1$. 
Then, for all $q\in \RR$, $a_{B,0}^T(q)=1$. 
This implies that $a_{n-1,0}^T(q)=n$. 
\end{remark}

Let $B \subseteq [n]$ with $|B| = r$.  Let $\sL_T^q[B|B]$ denote the 
$r \times r$ 
submatrix of $\sL_T^q$ induced on the rows and columns indexed by $B$.
From Corollary \ref{cor:smaller_dets}, we get $\det(\sL_T^q[B|B]) 
\geq 0$ when $B \neq [n]$.
When $B = [n]$, Bapat, Lal and Pati \cite{bapat-lal-pati} have shown 
that $\det(\sL_T^q) = 1-q^2$.
As remarked in Section \ref{sec:intro}, when $q \in \RR$ with $|q| \leq 1$, 
the matrix $\sL_T^q$ is positive semidefinite.

\begin{remark}
\label{rem:sturm_thm}
By Sturm's Theorem (see \cite{godunov-modern-linear-algebra}), the number of 
negative eigenvalues of $\sL_T^q$ equals the number of sign changes among
the leading principal minors.   When $|q| > 1$,  the number of sign
changes equals 1 by Corollary \ref{cor:smaller_dets}.
This gives a short proof of a result of Bapat, Lal and Pati 
\cite[Proposition 3.7]{bapat-lal-pati} that the signature of $\sL_T^q$ is 
$(n-1,1,0)$ when $|q| > 1$, where signature of a Hermitian matrix
$A$ is the vector $(p,n,z)$ with $p,n$ being the number of 
positive, negative eigenvalues of $A$  respectively
and $z$ being the nullity of $A$.
\end{remark}

\begin{remark}
By \eqref{eq:def_a_B,i},
for any $T$, all $a_{B,i}^T(q) \in \RR^+[q]$
when $B \not= [n]$. In \cite[Corollary 13]{mukesh-siva-hook}, 
it was shown that $a_{[n],i}^T(q) \in \RR^+[q]$ 
when $i > 0$.  By definition, 
$a_{[n],0}^T(q) = 1-q^2$ has negative coefficients.  
In \cite[Theorem 2.4]{mukesh-siva-hook}, it was shown
that $c_{\lambda,n}^{\sL_T^q} \in \RR^+[q]$ for all
$\lambda \vdash n$ except $\lambda = 1^n$.

By these and  Corollary \ref{cor:coeff_interpret}, it is easy to see 
that 
barring $c_{1^n,n}^{\sL_T^q}(q)$, which equals $1-q^2$,
$c_{\lambda,r}^{\sL_T^q}(q) \in \RR^+[q]$
for all $\lambda \vdash n$ and for all $0\leq r \leq n$.
Thus all statements in this work can be made about 
$c_{\lambda,r}^{\sL_T^q}(q)$ or alternatively about    
the {\sl absolute value} of the coefficient of $x^{n-r}$ 
in $f_{\lambda}^{\sL_T^q}(x)$
(which equals $(-1)^r c_{\lambda,r}^{\sL_T^q}(q)$). 
\end{remark}

\section{Proof of Theorem \ref{thm:main} }
\label{sec:coeff_imman_poly}

We begin with a few preliminaries towards proving Theorem \ref{thm:main}.
Let $T_1$ and $T_2$ be trees on $n$ vertices with $T_2 \geq_{\GTS_n} T_1$.  We 
assume that both $T_1$ and $T_2$ have vertex set  $V=[n]$.

\begin{remark}
  \label{rem:label}
Since immanants are invariant under a relabelling of vertices 
(see Littlewood's book \cite{littlewood-book} or Merris
\cite{merris-immanantal_invariants}),
without loss of generality, we label the vertices of $T_1$ as follows:
first label the vertices on the path $P_k$ as $1,2,\ldots, k$ in order with 
$1$ being the closest vertex to $X$ and $k$ being the closest vertex 
to $Y$.   Then, 
label vertices in $X$ with labels $k+1, k+2, \ldots, k+|X|$ in increasing order 
of distance from vertex 1 (say in a breadth-first manner starting from vertex 1) 
and lastly, label vertices of $Y$ from $n-|Y|+1$ to $n$ again in increasing order
of distance from vertex 1.  See Figure 
\ref{fig:gts_label_example} for an example.
\end{remark}

\begin{figure}[h]
\centerline{\includegraphics[scale=0.55]{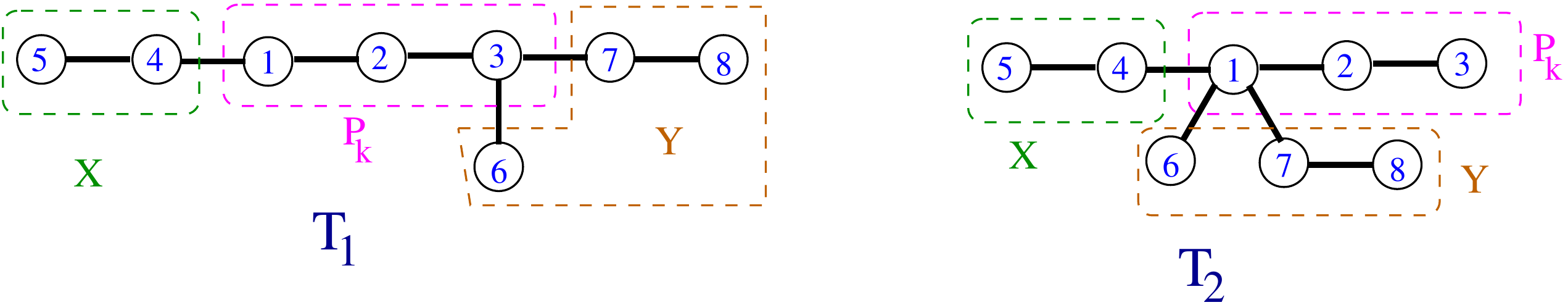}}
\caption{Two labelled trees with $T_2 \geq_{\GTS_n} T_1$ and $T_2$ covering $T_1$.}
\label{fig:gts_label_example}
\end{figure}

Recall our notation 
$a_{B,i}^{T_1}(q)$ and $a_{B,i}^{T_2}(q)$ for the trees $T_1$ and $T_2$ respectively.
Also recall $\sO_{B,i}^{T_1}$ denotes the set 
of $B$-orientations in $T_1$ with
$i$ bidirected-arcs and let $\sO_{r,i}^{T_1} = \cup_{B \subseteq V, |B|=r}^{ } \sO_{B,i}^{T_1}$. 
Recall that $\sO_{r,i}^{T_2}$ is defined analogously. 
It would have been nice if for all $B \subseteq V$ with $|B| = r$ and for 
all $0 \leq i \leq \rhalf$, we could prove that $a_{B,i}^{T_1}(q) - a_{B,i}^{T_2}(q)
\in \RR^+[q^2]$.  Unfortunately, this is not true as the example below
illustrates.

\begin{example}
  Let $T_2$ and $T_1$ be the trees given in Figure 
\ref{fig:gts_label_example}.  Let
  $B = \{1,4,6,7,8\}$ and let $i=2$.  It can be checked that 
  $a_{B,i}^{T_2}(q) = 2q^4+q^6$ and that $a_{B,i}^{T_1}(q) = q^4$.
\end{example}

Nonetheless, by combining all sets $B$ of size $r$, we will for all $r$, $i$ 
construct an injective map $\gamma : \sO_{r,i}^{T_2} \rightarrow \sO_{r,i}^{T_1}$ 
that preserves the ``away'' statistic.  
For each $r$, note that there are $\binom{n}{r}$ sets $B$ that 
contribute to $\sO_{r,i}^{T_2}$ and $\sO_{r,i}^{T_1}$.
We partition the $r$-sized subsets $B$ into three disjoint 
families and apply three separate lemmas.
Recall that vertices $1$ and $k$ are the endpoints of the path $P_k$ used
in the definition of the poset $\GTS_n$.
The first family consists of those sets $B$ with both
$1,k \not\in B$. 
\begin{lemma}
\label{lem:1_and_k_not_in_O}
Let $B \subseteq [n]$, $|B| = r$ be such that both $1,k \not \in B$.  
Then, there is an injective map $\phi: \sO_{B,i}^{T_2} \rightarrow 
\sO_{B,i}^{T_1}$ such that $\awy_B^{T_2}(O) = \awy_B^{T_1}(\phi(O))$.  
Thus, for all $0 \leq i \leq \rhalf$, 
we have $a_{B,i}^{T_1}(q) - a_{B,i}^{T_2}(q) \in \RR^+[q^2]$.
\end{lemma}
\begin{proof}
  Let $O \in \sO_{B,i}^{T_2}$.  Clearly, $1= \min_{u \in [n]-B}^{ } u$ and
for $O$, define $O' = \phi(O)$ as follows.  In $O'$, for each vertex
$v \in B$, assign the same orientation as in $O$.  Clearly, 
$O' \in \sO_{B,i}^{T_1}$  and it is clear that $\phi$ is  
an injective map from $\sO_{B,i}^{T_2}$ to $\sO_{B,i}^{T_1}$.
Further, it is easy to see that $\awy_B^{T_2}(O) = \awy_B^{T_1}(\phi(O))$, hence 
proving that $a_{B,i}^{T_1}(q) - a_{B,i}^{T_2}(q) \in \RR^+[q^2]$, completing the proof.
\end{proof}

\vspace{2 mm}

We next consider those $B$ with $|\{1,k\} \cap B| = 1$.  We use the 
notation $B$ for $r$-sized subsets with $1 \in B, k \not\in B$ and 
$B'$ for $r$-sized subsets with $k \in B', 1 \not\in B'$.  
The next lemma below considers such subsets $B'$ and those 
$B$-orientations $O$ with $O(1) \in X \cup P_k$.  Note that
for such $B$-orientations $O$, $\min_{v \in [n]-B}^{ }v \in P_k$.

\begin{lemma}
  \label{B'_and_B_easy_case}
Let $O \in \sO_{B,i}^{T_2}$, where $1 \in B, k \not\in B$ and let $O(1)$ 
denote the oriented neighbour 
of vertex 1 in $O$.  If $O(1) \in X \cup P_k$, then there exists 
an injective map $\mu: \sO_{B,i}^{T_2} \rightarrow \sO_{B,i}^{T_1}$ such
that $\awy_B^{T_2}(O) = \awy_B^{T_1}(\mu(O))$.  
Similarly, let $B' \subseteq V$ be such that $1 \not\in B', k \in B'$.   Then, 
there is an injective map
$\nu: \sO_{B',i}^{T_2} \rightarrow \sO_{B',i}^{T_1}$ such that for
$P \in \sO_{B',i}^{T_2}$, $\awy_{B'}^{T_2}(P) = \awy_{B'}^{T_1}(\nu(P))$.
\end{lemma}
\begin{proof}
The proof for both cases are similar.
Let $O \in \sO_{B,i}^{T_2}$ and let $O(1) \in X \cup P_k$.
In this case, the same injection of Lemma \ref{lem:1_and_k_not_in_O}
works.  That is, we form $O'$ by assigning all vertices of $B$ the 
same orientation as in $O$.  Clearly, $O' \in \sO_{B,i}^{T_1}$  and
$\awy_B^{T_2}(O) = \awy_B^{T_1}(O')$.

Similarly, let $P \in \sO_{B',i}^{T_2}$.  Form 
$P' \in \sO_{B',i}^{T_1}$ 
by assigning all vertices of $B'$ the same orientation as in $P$.
Clearly, 
$\awy_{B'}^{T_2}(P) = \awy_{B'}^{T_1}(P')$.  Note that in both $P$ and $P'$,
the orientation of $k$ equals $k-1$ as $k$ is a leaf vertex in $T_2$.  The proof is complete.
\end{proof}

\vspace{2 mm}

We continue to use the notation $B$ for an $r$-sized subset of $V$ with
$1 \in B$.  We now handle $B$-orientations $O \in \sO_{B,i}^{T_2}$ with $O(1) \in Y$.

\begin{lemma}
\label{lem:mapping_B_to_B'}
Let $B$ be an $r$-sized subset of $[n]$ with $1 \in B$, $k \not \in B$.  Define
$B' = (B - \{1\}) \cup \{k\}$.  Let $O \in \sO_{B,i}^{T_2}$
with $O(1) \in Y$.  There is an injective map $\delta: \sO_{B,i}^{T_2} 
\rightarrow \sO_{B',i}^{T_1}$ such that 
$\awy_B^{T_2}(O) = \awy_{B'}^{T_1}(\delta(O))$.  Futher, if
$N = \delta(O)$, then we have $N(k) = O(1)$.
\end{lemma}
\begin{proof}
The proof is identical to the proof of \cite[Lemma 7]{mukesh-siva-hook}.  We hence
only sketch our proof.  In $T_1$, define
$m' = \min_{v \in [n] - B'}v$ and recall that
$m = \min_{v \in [n] - B}v$ in $T_2$.
Since $1 \not \in B'$, note that in $T_1$, we have $m'=1$.  
Thus, we
reverse the orientation of some vertices in $T_2$ on the subpath 
from $(1,m)$ of $P_k$.  To 
decide the vertices whose orientations are to be reversed, we break the $(1,m)$ path
into segments separated by bidirected arcs. 
In each segment, 
if the $\ell$-th closest vertex to $m$ in $T_2$ was oriented ``towards $m$'', then 
in $T_1$, orient the $\ell$-th closest vertex to $1$ ``towards 1''.  Likewise, 
if the $\ell$-th closest vertex to $m$ in $T_2$ was oriented ``away from $m$'', 
then in $T_1$, orient the $\ell$-th closest vertex to $1$ ``away from 1''.

See Figure \ref{fig:case0} for an example, where the letter
``t'' is used to denote a vertex whose orientation is towards $m$ and
``a'' is used to denote a vertex whose orientation is away from $m$.  
This convention of ``t'' and ``a'' will be used in later figures as well.
For the example in the Figure \ref{fig:case0}, note that $k = 9$.
If $\delta$ is the map described above,
then it is clear that $\awy_B^{T_2}(O) = \awy_{B'}^{T_1}(\delta(O))$ and 
that $(\delta(O))(k) = O(1)$.
The proof is complete.
\end{proof}

\begin{figure}[h]
\centerline{\includegraphics[scale=0.5]{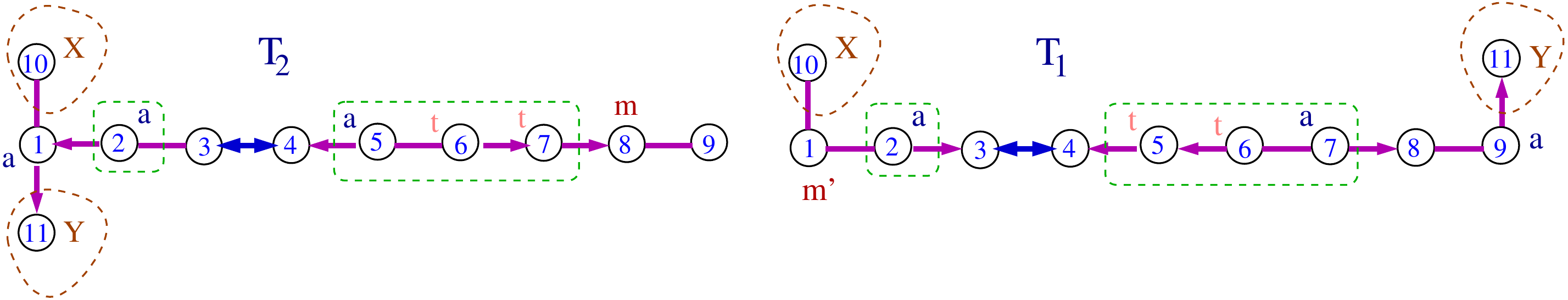}}
\caption{Illustrating the injection when $O(1) \in Y$, $m \in P_k$.}
\label{fig:case0} 
\end{figure}

\begin{corollary}
  \label{cor:mix_1_and_k}
Let $B \subseteq V$ with $1 \in B, k \not\in B$ and define 
$B' = (B - \{1\}) \cup \{k\}$.  For all
$i$, there is an injection $\omega: \sO_{B,i}^{T_2} \cup \sO_{B',i}^{T_2} \rightarrow
\sO_{B,i}^{T_1} \cup \sO_{B',i}^{T_1}$.
Thus, $a_{B,i}^{T_1}(q) + a_{B',i}^{T_1}(q) - a_{B,i}^{T_2}(q) -
a_{B',i}^{T_2}(q) \in \RR^+[q^2]$.
\end{corollary}
\begin{proof}
If $O \in \sO_{B,i}^{T_2}$ is such that $O(1) \in X \cup P_k$, use Lemma
\ref{B'_and_B_easy_case}. On the other hand,  if $O(1) \in Y$, then 
we use Lemma \ref{lem:mapping_B_to_B'}.  Let $O' = \omega(O)$.
Note that in this case, vertex $k$ is oriented with $O'(k) \in Y$.

Similarly, if $O \in \sO_{B',i}^{T_2}$, then, 
we use Lemma  \ref{B'_and_B_easy_case}.  Note that in this
case if $O' = \omega(O)$, then $O'(k) = k-1 \in P_k$.  Thus,
the case mentioned in the earlier paragraph and this case
are disjoint and hence $\omega$ is an injection.
\end{proof}

\vspace{2 mm}

Our last family consists of subsets $B$ with both $1,k \in B$.  
Define another subset $B' \subseteq [n]$ using $B$ as follows:  
Let $B_{xy} = B\cap (X \cup Y)$ and let $B_p = B \cap P_k$.   
The set $B'$ will be used when $m \in P_k$.  In this case,
$m = \min_{v \in P_k, v \not \in B} v$ is the mimimum
vertex outside $B$ in $P_k$.  
Define $l = \max_{v \in P_k, v \not \in B} v$ to be the maximum
numbered vertex in $P_k$ not in $B$.   Define $m'= k+1-l$ and
$l' = k+1-m$.   Form $B_p^t$ by taking the union of the three
sets $A' = \{1,\ldots, m'-1\}$, $C' =\{l'+1, \ldots, k \}$ and 
$\{m'-m +x : x \in B \cap \{m+1,\ldots,l-1\}  \}$.  See 
Figure \ref{fig:case8} for an example.  
Define $B' = B_{xy} \cup B_p^t$.  Clearly, both 
$1,k \in B'$ and $(B')' = B$.

\begin{lemma}
\label{lem:1_and_k_in_O}
Let $B \subseteq [n]$  be such that both $1,k \in B$ and let $B'$ be as defined
above.  For all $i$, there 
is an injective map $\theta: \sO_{B,i}^{T_2}  \cup \sO_{B',i}^{T_2} \rightarrow 
\sO_{B,i}^{T_1} \cup \sO_{B',i}^{T_1}$ that preserves the
away statistic.
Thus,  $a_{B,i}^{T_1}(q) + a_{B',i}^{T_1}(q) - a_{B,i}^{T_2}(q) -a_{B',i}^{T_2}(q) 
\in \RR^+[q^2]$.
\end{lemma}

\begin{proof} 
We denote the orientation of vertex 1 in $O$ as $O(1)$.   
Given $B$, recall $m = \min_{v \not\in B} v$ is the minimum vertex outside $B$ and
that we have labelled vertices on the path $P_k$ first, vertices in $X$ next and
vertices of $Y$ last.  
There are nine cases based on $m$ and $O(1)$.  
Only one of the nine cases will involve $B$ getting changed to $B'$.  For 
now, let $O \in \sO_{B,i}^{T_2}$.
Define a map $\theta: \sO_{B,i}^{T_2} \rightarrow \sO_{B,i}^{T_1}$ as follows.
Let $O \in \sO_{B,i}^{T_2}$.  
We construct a unique 
$O' \in \sO_{B,i}^{T_1}$  by using the algorithms tabulated below.  
Though it seems that there are a large number of cases, the underlying
moves are very similar.

For vertices $u,v,a,b$, we explain an operation that we denote as $\revpath(u,v; a,b)$ 
that will be needed when $m \in Y$.  We will always have $d_{u,v} = d_{a,b}$ in 
$T_1$ where $d_{u,v}$ is the distance between vertices $u$ and $v$ in $T_1$.
Further, all vertices $w$ on the $u,v$ path $P_{u,v}$ in $T_1$ will 
be in $B$ and hence be oriented.
$\revpath(u,v;a,b)$ will change orientations of all vertices on $P_{u,v}$.  
We will use this operation in all the three cases  when $m \in Y$.
Due to our labelling convention 
and the fact that  $m \in Y$, all vertices of $P_k \cup X$ will be 
contained in $B$.  In $T_2$, vertex
$m$ has vertex 1 as its closest vertex among the vertices in $P_k$, whereas 
in $T_1$, vertex $m$ has vertex $k$ as its closest vertex among those 
in $P_k$.  Denote vertices on $P_{u,v}$ as $u=u_1,u_2,\ldots,u_s=v$ and
the vertices on the $(a,b)$ path as $a=a_1,a_2,\ldots,a_s=b$.  
In $O$, if vertex $a_i$ is oriented ``towards $m$'', then orient vertex 
$u_{s+1-i}$ ``towards $m$'' and likewise if vertex $a_i$ is oriented 
``away from $m$'', then orient vertex $u_{s+1-i}$ ``away from $m$''.
We give the map $\theta$ using several algorithms below.

$$ 
\begin{array}{| c | c | c | c |} 
\hline
& m \in P_k & m \in X & m \in Y \\ \hline
O(1)=2 \in P_k & \mbox{Use algorithm 1}  & \mbox{Use algorithm 1} & \mbox{Use algorithm 2} \\ \hline
O(1)=x \in X &  \mbox{Use algorithm 1}  & \mbox{Use algorithm 1} & \mbox{Use algorithm 4} \\ \hline
O(1)=y \in Y &  \mbox{Use algorithm 5}  & \mbox{Use algorithm 3} & \mbox{Use algorithm 2} \\ \hline
\end{array}
$$

\textbf{Algorithm 1: }  This is a trivial copying algorithm.  Define
$O' = \theta(O)$ with $O' \in \sO_{B,i}^{T_1}$  as follows.
In $O'$, retain the same orientation for all vertices $v \in B$.  
It is clear that $\awy_B^{T_2}(O) = \awy_B^{T_1}(O')$.

\bigskip
\textbf{Algorithm 2: } Since $m \in Y$, by our labelling convention, this means all the 
vertices of $P_k$ and $X$ are in $B$.  Form $O' = \theta(O)$ with 
$O' \in \sO_{B,i}^{T_1}$ by first copying the 
orientation $O$ for
each vertex.  Then perform $\revpath(1,k;1,k)$.  This is illustrated in Figure \ref{fig:case2}
when $O(1) = 2$ and $m \in Y$ and in Figure \ref{fig:case2prime} when both 
$O(1), m \in Y$.
It is clear that $\awy_B^{T_2}(O) = \awy_B^{T_1}(O')$.

\begin{figure}[h]
\centerline{\includegraphics[scale=0.5]{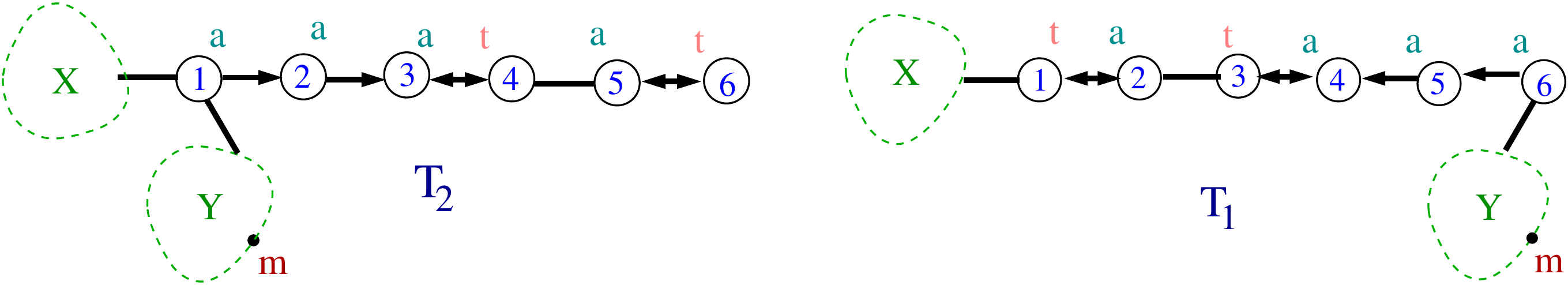}}
\caption{Illustrating Algorithm 2  with $O(1) \in P_k$, $m \in Y$.}
\label{fig:case2} 
\end{figure}

\begin{figure}[h]
\centerline{\includegraphics[scale=0.5]{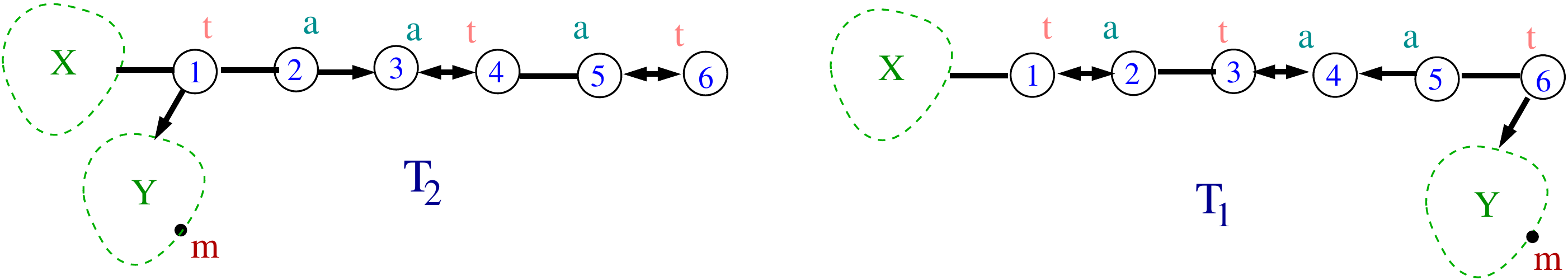}}
\caption{Illustrating Algorithm 2 with both $O(1),m \in Y$.}
\label{fig:case2prime} 
\end{figure}

\bigskip
\textbf{Algorithm 3: } We have $m \in X$ and $O(1) \in Y$.  Recall that we have labelled 
the vertices of $X$ in increasing order of distance from vertex 1.  We claim that 
there exists a unique edge 
$e = \{x,y\}$ on the path from $1$ to $m$ satisfying the  following two conditions:

\begin{enumerate}
  \item There is no arrow on $e$.  That is, either both $x,y \in B$ with $O(x) \neq y$ and 
	 $O(y) \neq x$ or $x \in B$ and $y = m$.
  \item Among such edges, $x$ is the closest vertex to $1$ 
	 distancewise (that is, $e$ is the unique closest edge to $1$).
\end{enumerate}

That there exists such an edge $e$ satisfying condition (1) above is easy to see.
Condition (2) is just a labelling of vertices of such an edge. 
Further, 
we label
the vertices on the path from $1$ to $x$ in increasing order of distance from
vertex $1$ as $1,x_1,x_2,\ldots,x_l=x$.  (See Figure 
\ref{fig:last_map_example} for an example.)

It is easy to see that $O(x_1) = 1$ and $O(x_i) = x_{i-1}$ for $2 \leq i \leq l$ 
and recall that $O(1) \in Y$.  
Form $O' = \theta(O)$ with $O' \in \sO_{B,i}^{T_1}$ as follows.  Vertices 
of $B$ not on the path 
from $x_l$ to $k$ in $T_1$ get the same orientation as in $O$.  
We orient the last $l+1$ vertices in $T_1$ on the $x_l$ to $k$ path 
$P_{x_l,k}$ away from  $m$,
and then orient the first $k-1$ vertices on $P_{x_l,k}$ as they 
were on $P_k$.  See Figure \ref{fig:last_map_example} for an example.   
As $k=6$ and $l=3$, the last $l+1$ vertices on the $(x_3,6)$ path
means that the last $4$ vertices are oriented away from $m$.  The
orientation of the remaining vertices is inherited from $T_2$.
It is clear that 
$|\bd(O)| = |\bd(O')|$ and that 
$\awy_B^{T_2}(O) = \awy_B^{T_1}(O')$.

\begin{figure}[h]
\centerline{\includegraphics[scale=0.55]{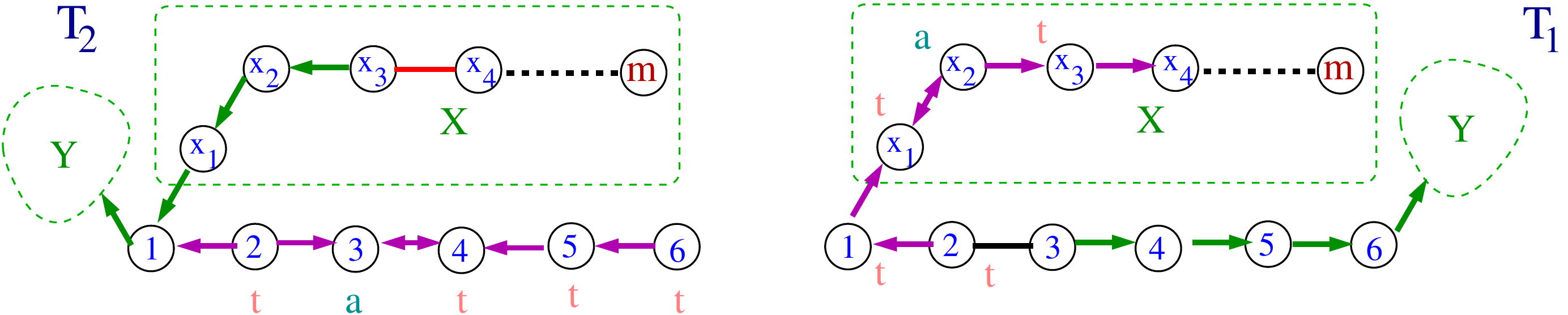}}
\caption{Illustrating Algorithm 3.}
\label{fig:last_map_example}
\end{figure}

\bigskip
\textbf{Algorithm 4:}  We have $O(1) \in X$ and $m \in Y$.  As done in
Algorithm 3, find the closest edge 
$e = \{x,y \}$ to vertex $1$ with $e$ having no arrow
on the $1$ to $m$ path.
As before, label $e$ as $\{x,y\}$ with $x$ being closer to $1$ than $y$, and label
the vertices on the path from $1$ to $x$ as $1,x_1,x_2,\ldots,x_l=x$  (see Figure 
\ref{fig:case6}).

It is easy to see that $O(x_1) = 1$ and $O(x_i) = x_{i-1}$ for $2 \leq i \leq l$. 
Note that there is a continuous string of $l+1$ vertices that are oriented away
from $m$.
Form $O' = \theta(O)$ with 
$O' \in \sO_{B,i}^{T_1}$ as follows.  Vertices of $B$ not on the path from $x_l$ to $k$ 
in $T_1$ get the same orientation as in $O$.  The closest $l+1$ vertices of $B$ on the path 
from $1$ to $x_l$ in $T_1$ get oriented away from $m$.  Denote the path comprising
the last $k-1$ vertices on the $(1, x_l)$-path as $P_{\ell}$.  Let $\alpha, \beta$
be the first and last vertices of $P_{\ell}$.  Perform $\revpath(\alpha, \beta; 2,k)$.
See Figure \ref{fig:case6} for an example.   
It is clear that $\awy_B^{T_2}(O) = \awy_B^{T_1}(O')$. 

\begin{figure}[h]
\centerline{\includegraphics[scale=0.55]{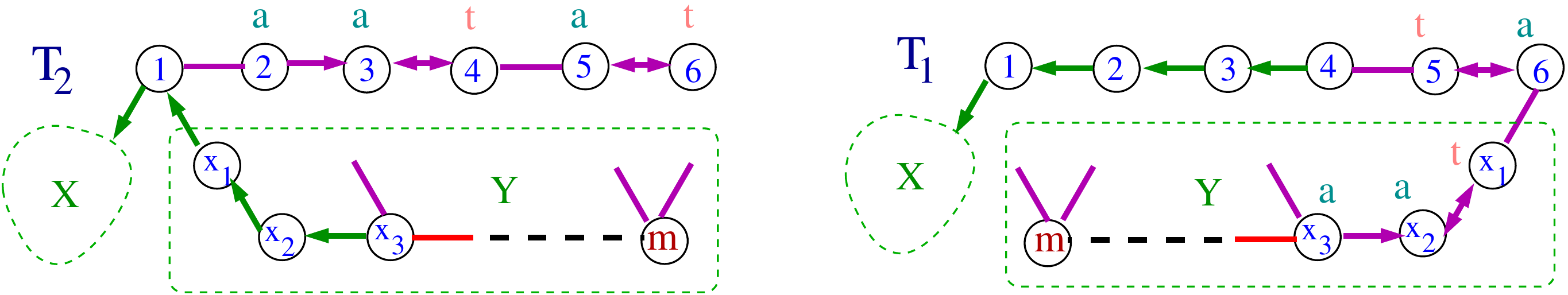}}
\caption{Illustrating Algorithm 4.}
\label{fig:case6}
\end{figure}

\bigskip
\textbf{Algorithm 5: } We have $O(1) = y \in Y$ and $m \in P_k$.
Recall $B' = B_{xy} \cup B_p^t$. 
Recall $l = \max_{v \in P_k, v \not \in B} v$.  Note that the minimum vertex 
$m' \not\in B'$ will be $m' = k+1-l$. 
Form $O' = \theta(O)$ with
$O' \in \sO_{B',i}$ as follows.  Note that in $T_2$, there is a 
continuous sequence $A$ of $m-1$ oriented vertices from $1$ to $m-1$ 
and another continuous sequence $C$ of $k-l$ oriented vertices from 
$l+1$ to $k$ in the path $P_k$ (see 
Figure \ref{fig:case8} for an example).  Similarly, in $T_1$, there is a 
continuous sequence $A'$ of $m'-1$ oriented vertices from $1$ to $m'-1$ 
and another continuous sequence $C'$ of $k-l'$ oriented vertices from 
$l'+1$ to $k$ in the path $P_k$.

It is easy to see that $|A| = |C'|$ and $|C| = |A'|$.
If vertex $s \in A$ is oriented away from (or towards) $m$ in $O$, then in
$O'$ orient vertex $k+1-s$ away from (or towards respectively) 
$m'$.  Likewise,
if vertex $s \in C$ is oriented away from (or towards) $m$ in $O$, then in
$O'$ orient vertex $k+1-s$ away from (or towards respectively) 
$m'$.  

Lastly, in $O'$ copy  the orientation of vertices 
in $B$ that lie between $m$ and $l$ in $T_2$ as they were to the 
vertices in $B'$ between $m'$ and $l'$ in $T_1$.  Formally,
if vertex $s \in P_k$ with $m < s < l$  is oriented
away from (or towards) $m$ in $O$, then in
$O'$ orient vertex $(m'-m)+s$ away from (or towards respectively) 
$m'$.   

For vertices $s \in B_p$, if 
See Figure \ref{fig:case8} for an example.  
Clearly, $|\bd(O)| = |\bd(O')|$ and 
$\awy_B^{T_2}(O) = \awy_{B'}^{T_1}(O')$.  This completes 
Algorithm 5.

\begin{figure}[h]
\centerline{\includegraphics[scale=0.55]{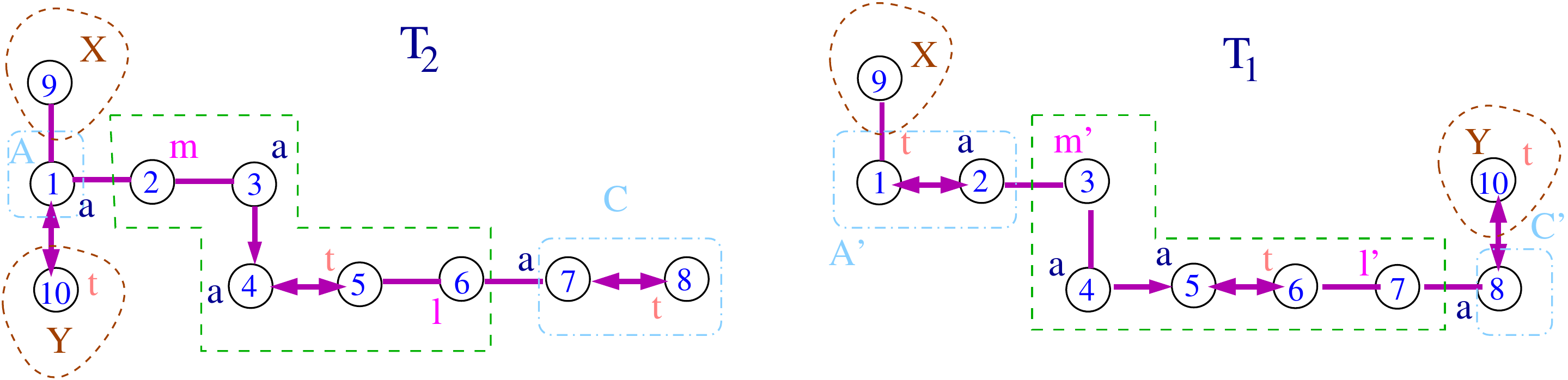}}
\caption{Illustrating Algorithm 5 with $B = \{1,3,4,5,7,8,10 \}$ and  $B' = \{1,2,4,5,6,8,10 \}$.  
}
\label{fig:case8}
\end{figure}

When $B = [n]$,  note that all vertices are oriented and hence 
there exists at least one bidirected edge.
In this case, 
we have $\awy_B^{}(O) = \aw(O,e)$, where as defined in 
\cite{mukesh-siva-hook}, $\aw(O,e)$ is found with respect to the lexicographic
minimum bidirected edge $e \in O$.  If the lexicographic edge is $e = \{u,v\}$,
we let $m = \min(u,v)$ be the smaller numbered vertex among $u,v$.  We find
the statistic $\aw(O,e)$
with respect to $m$.  It is simple to note that
among the nine cases, the following will not occur when $B = [n]$ due to
our labelling convention:  

\noindent
$(1)$ $m \in X$ and $O(1) \in P_k$, \hfill  
$(2)$ $m \in Y$ and $O(1) \in P_k$ and \hfill
$(3)$ $m \in Y$ and $O(1) \in X$.

In the remaining cases, we follow the same algorithms.
It is easy to see that the pair $(m, O(1))$ is different in all the  nine cases.
We do not change $B$ in eight cases, except in Algorithm 5.   Thus, 
we get an injection in these eight cases.  
When Algorithm 5 is run,
we get an injection from $\sO_{B,i}^{T_2}$ to $\sO_{B',i}^{T_1}$
and similarly we get an injection from
 $\sO_{B',i}^{T_2}$ to $\sO_{B,i}^{T_1}$.  Thus,
we get an injection from 
$\sO_{B',i}^{T_2} \cup \sO_{B,i}^{T_2}$ to $\sO_{B,i}^{T_1} 
\cup \sO_{B',i}^{T_1}$, completing the proof.
\end{proof}

\vspace{2 mm}

With these Lemmas in place, we can now prove Theorem \ref{thm:main}.

\vspace{2 mm}

\begin{proof} (Of Theorem \ref{thm:main}) 
We group the set of $r$-sized subsets $B$ into three categories:
those without $1,k$, those with either $1$ or $k$ and those
with both $1,k$.
By Lemmas \ref{lem:1_and_k_not_in_O}, \ref{lem:1_and_k_in_O} and 
Corollary \ref{cor:mix_1_and_k} it is clear that 
there is an injective map from 
$\sO_{r,i}^{T_2}$ to  $\sO_{r,i}^{T_1}$ for all $r$ and $i$.  By 
Corollary \ref{cor:coeff_interpret},
$c_{\lambda,r}^{\sL_{T_2}^q}(q) - c_{\lambda,r}^{\sL_{T_1}^q}(q) \in \RR^+[q^2]$
for all $\lambda, r$.  
\end{proof}

\begin{corollary}
Setting $q=1$ in $\sL_T^q$, we infer that for all $r$, the coefficient
of $x^{n-r}$ in the 
immanantal polynomial of the Laplacian $L_T$ of $T$ decreases 
in absolute value as we go up $\GTS_n$.
Using Lemma \ref{lem:csikvari-prelim}, we thus get a 
more refined and hence stronger result 
than Theorem \ref{thm:chan_lam_yeo}.
\end{corollary}

\begin{corollary}
\label{cor:immanant_q-Laplacian}
Let $T_1,T_2$ be trees on $n$ vertices with respective $q$-Laplacians 
$\sL_{T_1}^q, \sL_{T_2}^q$.   Let $T_2\geq_{\GTS_n} T_1$ and 
let $d_{\lambda}(\sL_{T_i}^q)$ denote the 
immanant of $\sL_{T_i}^q$ for $1\leq i \leq 2$ corresponding to the 
partition $\lambda \vdash n$.  By comparing the constant term of the
immanantal polynomial, for all $\lambda \vdash n$, 
we infer $d_{\lambda}(\sL_{T_2}^q) 
\leq  d_{\lambda}(\sL_{T_1}^q)$.  This refines the inequalities
in Theorem \ref{thm:chan_lam_yeo}.
\end{corollary}

\section{$q^2$-analogue of vertex moments in a tree}
\label{sec:centroid}
Merris in \cite{merris-second-imm-polynom}
gave an alternate definition of the centroid of a tree $T$ 
through its vertex moments.  He then  showed that
the sum of vertex moments appears as a coefficient 
of the immanantal polynomial of $L_T$ corresponding to
the partition $\lambda = 2,1^{n-2}$.
In this section, we define a $q^2$-analogue of vertex moments and 
through it, 
the centroid of a tree.   We then show that the sum of $q^2$-analogue
of the vertex moments of all vertices appears as a coefficient 
in the second immanantal polynomial of $\sL_T^q$.  
Thus, by Theorem \ref{thm:main},
the sum of the $q^2$-analogue of vertex 
moments decreases as we go up on $\GTS_n$.  
We further show that as we go up on $\GTS_n$,  
the value of the minimum $q^2$-analogue of the vertex moments also 
decreases. 

The following definition of vertex moments is from Merris 
\cite{merris-second-imm-polynom}.
Let $T$ be a tree with vertex set $[n]$.  For a vertex $i \in [n]$,
define 
$\ssM^T(i) = \sum_{j \in [n]} d_j d_{i,j}$ where $d_j$ is 
the degree of vertex $j$ in $T$ and $d_{i,j}$ is the distance
between vertices $i$ and $j$ in $T$.   Define 
the $q^2$-analogue of the distance $d_{i,j}$ between vertices
$i$ and $j$ to be 
$[d_{i,j}]_{q^2} =1+q^2+(q^2)^2+\cdots+(q^2)^{d_{i,j}^{ }-1}$ 
and define for all $i \in [n]$, $[d_{i,i}]_{q^2} = 0$.
We define the $q^2$-analogue of the moment of vertex $i$ 
of $T$ as
\begin{equation}
\label{eq:1def_centroid}
\ssM_{q^2}^T(i)=\sum\limits_{j\in [n]} [1+q^2(d_j-1)] [d_{i,j}]_{q^2}.
\end{equation}
Fix $q\in \RR$, $q \not= 0$.  Vertex $i$ is called the centroid 
of $T$  if $\ssM_{q^2}^T(i)=\min_{j\in [n]} \ssM_{q^2}^T(j)$.  
We clearly recover Merris' definition of moments when we plug in $q=1$ in 
\eqref{eq:1def_centroid}.
Merris showed that his definition of centroid coincides with 
the usual definition of the centroid of a tree $T$.
In \cite{bapat-siva-third-immanant},
Bapat and Sivasubramanian while studying the third immanant of $\sL_T^q$ 
proved a lemma that we need.  The following lemma is obtained 
by setting $s = q^2$ in \cite[Lemma 3]{bapat-siva-third-immanant}.

\begin{lemma}[Bapat and Sivasubramanian]
  \label{lem:frm_third_imm}
Let $T$ be a tree with vertex set $V = [n]$ and let $i \in [n]$.  Then,  
\begin{equation}
  \label{eqn:third_imm}
\sum_{j \in [n]} q^2(d_j -1) [d_{i,j}]_{q^2} = 
\sum_{j \in [n]}[d_{i,j}]_{q^2} - (n-1).
\end{equation}
\end{lemma}

The following alternate expression for $\ssM_{q^2}^T(i)$ is easy 
to derive using Lemma \ref{lem:frm_third_imm} and the definition
\eqref{eq:1def_centroid}.  As the proof
is a simple manipulation, we omit it.

\begin{lemma}
\label{lem:eqv_def}
Let $T$ be a tree with vertex set $[n]$ and let $i \in [n]$.  Then,
\begin{equation}
  \label{eqn:more_moment}
\ssM_{q^2}^T(i)=(n-1)+2q^2 \sum_{j \in [n]} (d_j-1) [d_{i,j}]_{q^2}.
\end{equation}
\end{lemma}

The following lemma gives an algebraic interpretation for
the $q^2$-analogue of vertex moments in $T$.

\begin{lemma}
\label{lem:alternate_def}
Let $T$ be a tree with vertex set $[n]$. 
Let  $i\in [n]$ be a vertex and let $B=[n]-\{i\}$. 
Then, 
\begin{equation}
  \label{eqn:moment}
\ssM_{q^2}^T(i)=(n-1)a_{B,0}^{T}(q)+2a_{B,1}^{T}(q).
\end{equation}
\end{lemma}
\begin{proof}
Clearly for $B=[n]-\{i\}$, 
we have a unique $B$-orientation $O\in \sO_{B,0}$ with $\awy_B^{T}(O)=0$.
This is the orientation where every vertex $j \in [n] - i$ gets oriented
towards $i$.  Thus $a_{B,0}^{T}(q)=1$.

We will show that 
$a_{B,1}^{T}(q)= q^2\sum_{j \in [n]}(d_j -1)[d_{i,j}]_{q^2}$
and appeal to \eqref{eqn:more_moment}.  
By \eqref{eqn:third_imm}, equivalently, we need to show that
\begin{equation*}
a_{B,1}^{T}(q)  =  \sum_{j \in [n]} [d_{i,j}]_{q^2} - (n-1) 
 =  q^2 \sum_{j \in [n], j \not= i} [d_{i,j}-1]_{q^2}.
\end{equation*}

Root the tree $T$ at the vertex $i$ and recall $B = [n] - \{i \}$.  
Thus $m = i$.
Let $O$ be a $B$-orientation with one bidirected arc $e = \{u,v\}$
where we label the edge $e$ such that $d_{i,v} = d_{i,u} + 1$.  
That is, $u$ occurs on the path from $i$ to $v$ in $T$.
Since $n-1$ vertices are oriented and one edge is bidirected,
there must be one edge without any arrows (when seen pictorially).
It is easy to see that all edges $f \in T$ not on the path $P_{i,u}$
from $i$ to $u$ must be oriented towards $i$.  Moreover, it is clear
that the edge $f$ without arrows must be on the path $P_{i,u}$.
Thus, our choice
lies in orienting vertices in $P_{i,u}$ such that one edge does
not get any arrows.  
Let $f = \{x,y\}$ with $x$ being on the path from $i$ to $y$ in $T$
($x$ could be $i$ or $y$ could be $u$).  Thus, there are $d_{i,u}-1$ 
choices for the edge $f$.  In $O$, clearly, all
vertices from $y$ till $u$ on the path $P_{i,u}$ must be oriented
away from $i$.  Hence the contribution of all such orientations 
will be $q^2 + q^4 + \cdots + q^{2d_{i,u}-2}$.
Thus vertex $u$ contributes $q^2 [d_{i,u}-1]_{q^2}$ to
$a_{B,1}^{ }(q)$.  Summing over all vertices $u$ completes the
proof.
\end{proof}

\begin{theorem}
  \label{thm:algebr_inv}
Let $T$ be a tree with vertex set $[n]$  and  $q$-Laplacian $\sL_T^q$. 
Let $\lambda=2,1^{n-2}\vdash n$. Then, 
$$c_{\lambda,n-1}^{\sL_T^q}(q)=\sum\limits_{i=1}^{n} \ssM_{q^2}^T(i).$$
\end{theorem}

\begin{proof}
Summing \eqref{eqn:moment} over all $B$ with cardinality $n-1$, we get 
\begin{equation*}
\sum\limits_{i=1}^{n} \ssM_{q^2}^T(i)  = (n-1)a_{n-1,0}^T(q) + 2 a_{n-1,1}^T(q) 
  =  c_{\lambda,n-1}^{\sL_T^q}(q)
\end{equation*}

where the last equality follows from Corollary \ref{cor:coeff_interpret} 
and Lemma \ref{lem:character_sum} with $k=2$.  The proof is complete.
\end{proof}

On setting $q=1$ in Theorem \ref{thm:algebr_inv}, we recover
Merris' result \cite[Theorem 6]{merris-second-imm-polynom}.
From Theorem \ref{thm:main} and Theorem \ref{thm:algebr_inv}, 
we get the following.

\begin{theorem}
\label{thm:moment_gts}
Let $T_1$ and $T_2$ be trees with $n$ vertices and let $T_2$ cover $T_1$ in $\GTS_n$.  
Then, $$ \sum\limits_{i=1}^{n} \ssM_{q^2}^{T_2}(i)  \leq \sum\limits_{i=1}^{n} \ssM_{q^2}^{T_1}(i).$$
\end{theorem}

Theorem \ref{thm:moment_gts} implies that the sum of
the vertex moments decreases as we go up on the poset $\GTS_n$.  
We next show that the minimum value of the $q^2$-analogue of 
vertex moments also decreases as we go up on $\GTS_n$.

\begin{lemma}
\label{lem:min_moment}
Let $T_1$ and $T_2$ be two trees with vertex set $[n]$ such that  $T_2$ covers  $T_1$ in ${\GTS_n}$.  
Then, for all $q\in \RR$, 
we have $\min_{i\in [n]} \ssM_{q^2}^{T_2}(i)\leq \min_{j\in [n]} \ssM_{q^2}^{T_1}(j).$ 
\end{lemma}
\begin{proof} 
Let $l\in [n]$ be the vertex in $T_1$ with  
$\ssM_{q^2}^{T_1}(l)=\min_{i\in [n]} \ssM_{q^2}^{T_1}(i)$. Let
$l\in X\cup Y\cup P_{\lfloor k / 2\rfloor}$ (see Figure 
\ref{fig:gts_example} for $X$,  $Y$ and $P_k=P_{x,y}$).
Here $P_{\lfloor k/2 \rfloor}$ is the path $P_k$ restricted
to the vertices $1,2,\ldots, \lfloor k/2 \rfloor$.
Then, using 
the fact that the distance $d^{T_1}_{x,y}  \geq d^{T_2}_{x,y}$ 
for all pairs  $(x,y)\in X\times Y$, we have  
$$\ssM_{q^2}^{T_1}(l)\geq \ssM_{q^2}^{T_2}(l)\geq \min_{i\in [n]} \ssM_{q^2}^{T_2}(i).$$
If $l\geq \lfloor k / 2\rfloor$ then  
$\ssM_{q^2}^{T_1}(l)\geq \ssM_{q^2}^{T_2}(k+1-l)\geq \min_{i\in [n]} \ssM_{q^2}^{T_2}(i).$ 
Thus we can find a vertex $i$ in $T_2$ such that 
$\ssM_{q^2}^{T_1}(l)\geq \ssM_{q^2}^{T_2}(i)$, completing the proof. 
\end{proof}

\begin{corollary}
  \label{cor:min_moment}
Let $T_1,T_2$ be two trees on $n$ vertices with $T_2\geq_{\GTS_n} T_1$. 
Then, for all $q \in \RR$, the minimum $q^2$-analogue of the vertex moments
of $T_2$ is less than the minimum $q^2$-analogue of the vertex 
moments of $T_1$.
\end{corollary}

An identical statement about the maximum $q^2$-analogue of vertex 
moments is not true as shown in the following example.
\begin{example}
\label{example:max_centroid}
Let $T_1,T_2$ be trees on the vertex set $[8]$ given in Figure \ref{fig:moments_at_q=1}. 
Vertices $2$ and $3$ are both centroid vertices in $T_1$, 
while in $T_2$, the centroid is vertex $1$. 
The $q^2$-analogue of their vertex moments are as follows:
$\ssM_{q^2}^{T_1}(2)=\ssM_{q^2}^{T_1}(3)=9+2q^2(7+3q^2) 
\mbox{ \rm{ and } } \ssM_{q^2}^{T_2}(1)=9+2q^2(2+q^2).$ 
The $q^2$-analogue of vertex moments of leaf vertices of $T_1$ and $T_2$ are as follows.
\begin{eqnarray*}
\ssM_{q^2}^{T_1}(i) & = & 
9+2q^2(8+5q^2+4q^4+3q^6) \mbox{ \rm{ for  } } i=5,6,7,8,9,10. \\
\ssM_{q^2}^{T_2}(i) & = & 
9+2q^2(8+2q^2+q^4) \mbox{ \rm{ for  } } i=5,6,7,8,9,10. \\
\ssM_{q^2}^{T_2}(4) & = & 
9+2q^2(8+7q^2+6q^4).
\end{eqnarray*}
When $q=1$, the moments of vertices of $T_1$ and $T_2$ 
are given in Figure \ref{fig:moments_at_q=1} alongside the vertices.
Clearly, when $q=1$, 
$\max_{j\in [8]}^{ } \ssM_{q^2}^{T_2}(j)=51 \not\leq 49=\max_{j\in [8]}^{ }\ssM_{q^2}^{T_1}(j).$
\begin{figure}[h]
\centerline{\includegraphics[scale=0.6]{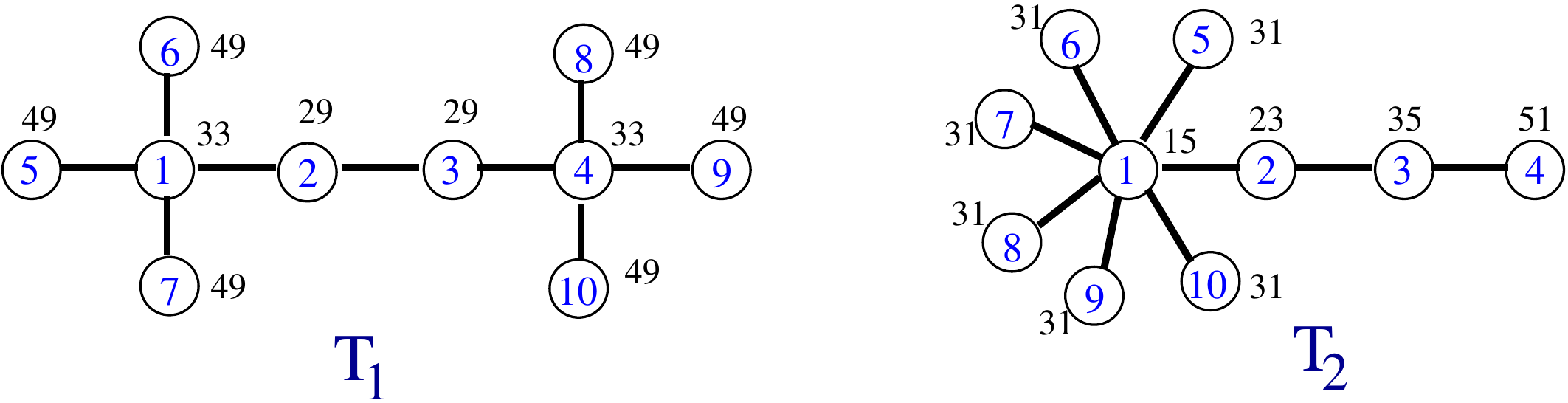}}
\caption{$q^2$-moments of vertices of $T_1$ and $T_2$ when $q=1$.}
\label{fig:moments_at_q=1}
\end{figure} 
\end{example}

Associated to a  tree are different notions of ``median'' and ``generalized
centers'', see the book \cite{kaul-mulder-book}.
It would be nice to see the behaviour of these parameters
as one goes up $\GTS_n$.

\section{$q,t$-Laplacian $\sL_{q,t}$ and Hermitian Laplacian of a tree $T$}
\label{sec:qt_laplacian}

All our results work for the bivariate Laplacian matrix $\sL_{q,t}$ 
of a tree $T$ on $n$ vertices defined as follows.  Let $T$ be a tree with edge
set $E$.  Replace each edge $e = \{u,v\}$ by two bidirected arcs, 
$(u,v)$ and $(v,u)$.  Assign one of the arcs, say $(u,v)$ a variable 
weight $q$ and its reverse arc, a variable weight $t$ and let 
$A_{n \times n} = (a_{i,j})_{1 \leq i,j \leq n}$ be the matrix
with $a_{u,v} = q$ and $a_{v,u} = t$.  Assign $a_{u,v} = 0$ if 
$\{u,v\} \not \in E$.
Let $D_{n \times n} = (d_{i,j})$ be 
the diagonal matrix with entries $d_{i,i} = 1+qt(\deg(i)-1)$.
Define $\sL_{q,t} = D - A$.  Note that when $q=t$, $\sL_{q,t} = \sL_T^q$
and that when $q=t=1$, $\sL_{q,t} = L_T$ where $L_T$ is the usual
combinatorial Laplacian matrix of $T$.

It is easy to see that our proof relies on the fact that the difference
in the coefficients of the immmanantal polynomial is a non-negative 
combination of the $a_{r,i}^T(q)$'s which are polynomials in $q^2$ and
that $q^2 \geq 0$ for all $q \in \RR$.  When $B = [n]$, bivariate
versions of $m_{n,j}(q,t)$ and $a_{n,i}^T(q,t)$ were defined in 
\cite{mukesh-siva-hook}.  Define bivariate versions $m_{r,j}(q,t)$ 
and $a_{r,i}^T(q,t)$ as done in Section \ref{sec:Bmatchings}
but replace all occurrences of $q^2$ with $qt$.

\comment{
\red{
\begin{remark}
Let $T$ be a tree on $n$ vertices with $q,t$-Laplacian  $\sL_{q,t}$. Then, 
 for $0\leq r \leq n$,
the coefficient of $(-1)^rx^{n-r}$ in the immanantal polynomial of $\sL_{q,t}$ 
indexed by $\lambda\vdash n$ equals $\sum_{i=0}^{\rhalf}2^ia_{r,i}^T(q,t) 
\alpha_{\lambda,i}^{}$.
\end{remark}
}}

With this definition, it is simple to see that all results go through for $\sL_{q,t}$, the 
$q,t$-Laplacian of $T$ whenever $q,t \in \RR$ and $qt \geq 0$ or
$q,t \in \CC$ and $qt \geq 0$.   One special case of $\sL_{q,t}$ is 
obtained when we set $q=\imath$ and $t=-\imath$ where 
$\imath = \sqrt{-1}$. In this
case, the weighted adjacency matrix becomes the Hermitian adjacency 
matrix of $T$ with edges oriented in the direction of the 
arc labelled $q$.  The Hermitian adjacency matrix is a matrix
defined and studied by Bapat, Pati and Kalita 
\cite{bapat-pati-kalita_weighted_digraphs} and later independently
by Liu and Li \cite{liu-li-hermitian-adj}
and by Guo and Mohar \cite{guo-mohar-hermitian-adj}.	
With these complex numbers as weights, $\sL_{q,t}$ reduces to what is defined as the 
Hermitian Laplacian of $T$ by Yu and Qu \cite{yu-qu-hermitian-Laplacian}.
We get the following corollary of Theorem \ref{thm:main}.

\begin{corollary}
  Let $T_1, T_2$ be trees on $n$ vertices with $T_2 \geq_{\GTS_n} T_1$. 
Then, in absolute value, the coefficients of the immanantal polynomials 
of the Hermitian Laplacian of $T_1$ are larger than the 
corresponding coefficient of the immanantal polynomials of
the Hermitian Laplacian of $T_2$. 
\end{corollary}

Let $T$ be a tree on $n$ vertices 
with Laplacian $L_T$ and $q,t$-Laplacian $\sL_{q,t}$.
When $q=z\in \CC$ with $z\neq 0$, and $t=1/q$ 
then it is simple to see  that the matrix $\sL_{q,t}$ need not be Hermitian. 
In this case, for all $i\geq 0$, we have $a_{r,i}^T(q)_{q=1}^{}=a_{r,i}^T(z,1/z)$. 
This implies that  
for all $\lambda \vdash n$ and for $0\leq r \leq n$,
$c_{\lambda,r}^{\sL_{q,t}}=c_{\lambda,r}^{L_{T}}$.
Thus, we obtain the following  simple corollary.  
\begin{corollary}
Let $T$ be a tree on $n$ vertices with Laplacian $L_T$ and $q,t$-Laplacian  $\sL_{q,t}$.
Then, for all  $z\in \CC$ with $z\neq 0$ and for all $\lambda \vdash n$
$$f_{\lambda}^{\sL_{z,1/z}}(x)=f_{\lambda}^{L_T}(x).$$
\end{corollary}

\section{Exponential distance matrices of a tree} 
\label{sec:expon_dist_mat}

In \cite{bapat-lal-pati}, Bapat, Lal and Pati introduced the 
exponential distance matrix $\ED_T$ of a tree $T$.
In this section, we prove that when $q \not= \pm 1$, 
the coefficients
of the characteristic polynomial of $\ED_T$, in absolute 
value decrease when we go up $\GTS_n$.  
We show a similar relation on immanants of $\ED_T$ indexed
by partitions with two columns.
We recall the definition of $\ED_T$ from \cite{bapat-lal-pati}.
Let $T$ be a tree with $n$ vertices.    
Then, its exponential distance matrix 
$\ED_T=(e_{i,j})_{1\leq i,j\leq n}$ is defined as follows: 
  the entry $e_{i,j}=1$ if $i=j$ and $e_{i,j}=q^{d_{i,j}}$ if $i\neq j$, 
  where $d_{i,j}$ is the distance between vertex $i$ and vertex $j$ in $T$.  
For $\lambda\vdash n$, define 
\begin{equation}
\label{eq:def_char_poly_ED}
f_{\lambda}^{\ED_{T}}(x) = d_{\lambda}(xI -\ED_{T}) = \sum_{r=0}^n (-1)^r c_{\lambda,r}^{\ED_{T}}(q) x^{n-r}.
\end{equation}

We need the following lemma of Bapat, Lal and Pati \cite{bapat-lal-pati}. 

\begin{lemma}[Bapat, Lal and Pati]
\label{lem:bapat_E_inverse}
Let $T$ be a tree with $n$ vertices. 
Let $\sL_{T}^q$ and $\ED_T$ be the $q$-Laplacian 
and exponential distance matrix of $T$ respectively. 
Then, $\det(\ED_T) = (1-q^2)^{n-1}$ and 
if $q\neq \pm 1$, then 
$$\ED_T^{-1}=\frac{1}{1-q^2} \sL_{T}^q.$$
\end{lemma}

Using Jacobi's Theorem on minors of the inverse of a matrix
(see DeAlba's article \cite[Section  4.2]{dealba:determinants}),
we get the following easy corollary, whose proof we omit.

\begin{corollary}
\label{cor:our_main_result_2}
Let $T$ be a tree with $n$ vertices. 
Let  $\sL_{T}^q$ and $\ED_T$ be the $q$-Laplacian 
and exponential distance matrix of $T$ respectively. Let $q\neq \pm 1$. 
Then, for $0\leq r \leq n$
 $$c_{1^n,r}^{\ED_{T}}(q)=(1-q^2)^{r-1}c_{1^{n},n-r}^{\sL_{T}^q}(q),$$
 where $c_{1^{n},n-r}^{\sL_{T}^q}(q)$ is the coefficient of 
 $(-1)^{n-r} x^{r}$ in $f_{1^n}^{\sL_{T}^q}(x).$
\end{corollary}

The following corollary is an easy consequence of Theorem \ref{thm:main} 
and  Corollary  \ref{cor:our_main_result_2}, we omit its proof. 

\begin{corollary}
Let $T_1$ and $T_2$ be two trees with $n$ vertices 
such that $T_2 \geq_{\GTS_n} T_1$. Then, for all $q\in \RR$ 
with $q \not= \pm 1$ and for $0\leq r \leq n$,
$$\left|c_{1^n,r}^{\ED_{T_2}}(q)\right| \leq \left|c_{1^n,r}^{\ED_{T_1}}(q)\right|.$$
In particular, for an arbitrary tree $T$ with $n$ vertices,
$$\left|c_{1^n,r}^{\ED_{S_n}}(q)\right| \leq \left|c_{1^n,r}^{\ED_{T}}(q)\right| 
\leq\left|c_{1^n,r}^{\ED_{P_n}}(q)\right|.$$
\end{corollary}

We give some results for  the immanant $d_{\lambda}(\ED_T)$, 
when $\lambda\vdash n$ is a two column partition. 
That is $\lambda=2^k,1^{n-2k}$ with $0 \leq k \leq \nhalf.$  
When $\lambda$ is a two column partition of $n$, 
Merris and Watkins in 
\cite{merris-watkins-ineqs_and_idents}
proved the following lemma for invertible matrices.

\begin{lemma}[Merris, Watkins]
\label{lem:merris_2_col}
Let $A$ be an  invertible $n \times n$  matrix.
Then $\lambda \vdash n$ is a two column partition if and only if 
$$d_{\lambda}(A)\det(A^{-1})=d_{\lambda}(A^{-1})\det(A).$$
\end{lemma} 

\begin{lemma}
  \label{lem:lapl-expo}
Let $T$ be a tree with $n$ vertices with 
$q$-Laplacian and  exponential distance matrices
$\sL_T^q$ and $\ED_T$  respectively. 
Then for all $q\in \RR$ with $q\neq \pm 1$ 
and $\lambda=2^k,1^{n-2k}$ for $0\leq k\leq \nhalf$
$$d_{\lambda}(\ED_T)=d_{\lambda}(\sL_T^q)(1-q^2)^{n-2}.$$ 
\end{lemma}
\begin{proof}
For all $q\in \RR$ with $q\neq \pm 1$, $\ED_T$ is invertible.
By Lemma \ref{lem:merris_2_col}, we have 
$$d_{\lambda}(\ED_T)\det\left(\frac{1}{1-q^2}\sL_T^q \right) 
 =  d_{\lambda}\left(\frac{1}{1-q^2}\sL_T^q \right)\det(\ED_T).$$

 $$\mbox{Thus, } d_{\lambda}(\ED_T) \det\left(\sL_T^q \right) =  d_{\lambda} \left(\sL_T^q \right)  \det(\ED_T). $$
Hence,  
$d_{\lambda}(\ED_T)=d_{\lambda}(\sL_T^q)(1-q^2)^{n-2}$, completing
the proof.
\end{proof}

Combining Lemma \ref{lem:lapl-expo} and Theorem \ref{thm:main} 
gives us another corollary whose straightforward proof we 
again omit.

\begin{corollary}
  \label{cor:immanant_ed_t}
Let $T_1$ and $T_2$ be two trees on $n$ vertices with 
$T_2\geq_{\GTS_n}T_1$. Then, for all $q\in \RR$ with 
$q\neq \pm 1$  and for all 
$\lambda=2^k,1^{n-2k}$, we have
$$|d_{\lambda}(\ED_{T_2})|\leq |d_{\lambda}(\ED_{T_1})|.$$
\end{corollary}

\subsection{$q,t$-exponential distance matrix}

We consider the bivariate exponential distance matrix in this subsection.
Orient the tree $T$ as done above.  Thus each 
directed arc $e$ of $E(T)$ has a unique reverse arc $e_{rev}$ and
we assign a variable weight $w(e) = q$ and $w(e_{rev}) = t$ or 
vice versa.  If the path
$P_{i,j}$ from vertex $i$ to vertex $j$ has the sequence of 
edges $P_{i,j} = (e_1,e_2, \ldots, e_p)$,  assign it weight
$w_{i,j} = \prod_{e_k \in P_{i,j}} w(e_k)$.  
Define $w_{i,i}=1$ for $i=1,2,\ldots,n$.
Define the bivariate exponential distance matrix 
$\ED_T^{q,t} = (w_{i,j})_{1 \leq i,j \leq n}$.
Clearly, when $q=t$, we have $\ED_T^{q,t} = \ED_T$.
Bapat and Sivasubramanian in \cite{bapat-siva-prod_dist_matrix} 
showed the following bivariate counterpart of Lemma 
\ref{lem:bapat_E_inverse}.

\begin{lemma}[Bapat, Sivasubramanian]
\label{lem:bapat_E_biv_inverse}
Let $T$ be a tree with $n$ vertices and
let $\sL_{T}^{q,t}$ and $\ED_T^{q,t}$ be its $q,t$-Laplacian 
and $q,t$ exponential distance matrix respectively. 
Then, $\det(\ED_T^{q,t}) = (1-qt)^{n-1}$ and 
if $qt \neq 1$, then 
$$(\ED_T^{q,t})^{-1}=\frac{1}{1-qt} \sL_{T}^{q,t}.$$
\end{lemma}

It is easy to see that all results about $\ED_T$ go through
for the bivariate $q,t$-exponential distance matrix $\ED_T^{q,t}$
when $q,t \in \RR$ with $qt \not= 1$ or when $q,t \in \CC$ with
$qt \not= 1$.  In particular, Corollary
\ref{cor:immanant_ed_t} goes through for the bivariate
exponential distance matrix.

\section*{Acknowledgement}
The first author acknowledges support from DST, New Delhi for 
providing a Senior Research Fellowship.
The second author acknowledges support from project grant
15IRCCFS003 given by IIT Bombay.

\bibliographystyle{acm}
\bibliography{main}
\end{document}